\documentclass[onefignum,onetabnum]{siamart220329}
\usepackage{amsmath}
\usepackage{amssymb}
\usepackage{mathrsfs}
\usepackage{mathtools}
\usepackage{amscd}
\usepackage[all]{xy}
\SelectTips{cm}{}
\usepackage{multirow}
\usepackage{stmaryrd}
\usepackage{enumitem}
\usepackage{textgreek}
\usepackage{hyperref}
\usepackage{cleveref}

\usepackage{tikz}
\tikzstyle{nodes1} = [circle, rounded corners, minimum width=1cm, minimum height=1cm,text centered, draw=black, fill=white!30]
\tikzstyle{arrow} = [thick,->,>=stealth]

\usepackage{xargs}                      
\usepackage{xcolor}  

\usepackage{algorithm}
\usepackage{algpseudocode}

\usepackage[colorinlistoftodos,prependcaption,textsize=tiny]{todonotes}
\newcommandx{\unsure}[2][1=]{\todo[linecolor=red,backgroundcolor=red!25,bordercolor=red,#1]{#2}}
\newcommandx{\change}[2][1=]{\todo[linecolor=blue,backgroundcolor=blue!25,bordercolor=blue,#1]{#2}}
\newcommandx{\info}[2][1=]{\todo[linecolor=OliveGreen,backgroundcolor=OliveGreen!25,bordercolor=OliveGreen,#1]{#2}}
\newcommandx{\improvement}[2][1=]{\todo[linecolor=Plum,backgroundcolor=Plum!25,bordercolor=Plum,#1]{#2}}
\newcommandx{\thiswillnotshow}[2][1=]{\todo[disable,#1]{#2}}

\newcommand{\lmd}{\lambda}
\newcommand{\pt}{\partial}
\newcommand{\re}{\mathbb{R}}
\newcommand{\st}{\mathit{s.t.}}
\newcommand{\ddd}{,\ldots,}
\newcommand{\be}{\begin{equation}}
\newcommand{\ee}{\end{equation}}

\newcommand{\deq}{\,\coloneqq\,}

\theoremstyle{plain}
\newsiamremark{remark}{Remark}
\newsiamremark{example}{Example}

\newcommand{\mc}[1]{\mathcal{#1}}
\newcommand{\highlight}[1]{\textcolor{black}{\bf{#1}}}

\def\deg{{\mathrm{deg}}}

\def\bN{{\mathbb N}}

\def\hm{{\hat{m}}}

\def\bI{{\bf{I}}}
\def\bG{{G}}

\def\ei{e}
\def\R{{\mathbb{R}}}

\def\Ideal{{\operatorname{Ideal}}}
\def\Qmod{{\operatorname{Qmod}}}
\def\IQ{{\operatorname{IQ}}}

\def\cC{{\mathcal{C}}}

\def\cA{{\mathcal{A}}}

\def\cP{{\mathcal{P}}}

\def\cR{{\mathcal{R}}}
\def\cQ{{\mathcal{Q}}}
\def\cD{{\mathcal{D}}}

\def\cI{{\mathcal{I}}}

\def\cJ{{\mathcal{J}}}
\def\cK{{\mathcal{K}}}

\def\deg{{\mathrm{deg}}}

\def\1{{\bf{1}}}

\def\rank{\operatorname{rank}}
\def\<#1,#2>{\langle #1,#2\rangle}
\def\<#1>{\langle #1\rangle}

\newtheorem{assumption}{Assumption}
\makeatletter
\newcommand{\extp}{\@ifnextchar^\@extp{\@extp^{\,}}}

\def\@extp^#1{\mathop{\bigwedge\nolimits^{\!#1}}}
\makeatother

\begin{document}

\title{A correlatively sparse Lagrange multiplier expression relaxation for polynomial optimization}

\author{Zheng Qu\thanks{Department of Mathematics, The University of Hong Kong, Pokfulam Road, Hong Kong. (email: {zhengqu@hku.hk})}
\and Xindong Tang\thanks{Department of Applied Mathematics, The Hong Kong Polytechnic University,  Hung Hom, Kowloon, Hong Kong. (email: {xdtang@hkbu.edu.hk})}}

\maketitle

\begin{abstract}
In this paper, we consider polynomial optimization with correlative sparsity. We construct correlatively sparse Lagrange multiplier expressions (CS-LMEs) and propose CS-LME reformulations for polynomial optimization problems using the Karush--Kuhn--Tucker optimality conditions. Correlatively sparse sum-of-squares (CS-SOS) relaxations are applied to solve the CS-LME reformulation. We show that the CS-LME reformulation inherits the original correlative sparsity pattern, and the CS-SOS relaxation provides sharper lower bounds when applied to the CS-LME reformulation, compared with when it is applied to  the original problem. Moreover, the convergence of our approach is guaranteed under mild conditions. In numerical experiments, our new approach usually finds the global optimal value (up to a negligible error) with a low relaxation order for cases where directly solving the problem fails to get an accurate approximation. Also, by properly exploiting the correlative sparsity, our CS-LME approach requires less computational time than the original LME approach to reach the same accuracy level.
\end{abstract}



\section{Introduction}
Let $n$ be a positive integer, and let $x:=(x_1,\ldots,x_n)$ be the variable in the $n$-dimensional Euclidean space.
Denote by $\re[x]$ be the ring of real coefficient polynomials in $n$ indeterminates.
We consider the polynomial optimization problem
\begin{equation}\label{eq:pb}
\left\{ \begin{array}{lll}
\displaystyle &\displaystyle \min_{x\in \R^{n}} &f(x)\\
\enspace& \mathrm{s.t.} \quad  &g(x)\geq 0, \quad  h(x)=0.
\end{array}\right.
\end{equation}
In the above, $f \in \R[x]$ is a polynomial,
and $g\in \R[x]^{m}$ and $h\in \R[x]^{\ell}$ are tuples of polynomial functions. 
In~\cite{Lasserre2001}, Lasserre introduced a hierarchy of semidefinite programming (SDP) relaxations to provide a sequence of lower bounds for~\eqref{eq:pb}, which converges to the global optimal value of~\eqref{eq:pb}, under some compactness assumptions. This approach is known as {\it the Moment-SOS relaxations} and has been intensively explored in the last two decades for global solutions of polynomial optimization problems. 
For (\ref{eq:pb}), Nie introduced the {\it Lagrange multiplier expressions} (LMEs) \cite{nie2019tight},
whose existence is guaranteed when $g(x)$ and $h(x)$ are given by generic polynomial functions.
LMEs can be applied to construct the {\it LME reformulation} of (\ref{eq:pb}) using the Karush--Kuhn--Tucker (KKT) optimality conditions,
which guarantees the moment relaxation being exact when the relaxation order is big enough and the global minimum for (\ref{eq:pb}) is attainable.
However, these approaches are usually computationally expensive.
Indeed, even for unconstrained polynomial optimization problems, i.e., $m=\ell=0$,
the moment relaxation for~\eqref{eq:pb} is an SDP problem with matrices of size up to $\binom{n+d}{n} \times \binom{n+d}{n} $, where $d\in\mathbb{N}$ is the relaxation order such that $2d\ge \deg(f)$. 

Given the polynomial optimization problem (\ref{eq:pb}),
let $(\mc{I}_1,\dots,\mc{I}_s)$ be subsets of $[n]:=\{1,\ldots,n\}$ such that $\bigcup_{i=1}^s\mc{I}_i=[n]$,
and denote $x^{(i)}:=(x_j)_{j\in\mc{I}_i}$.
Equation (\ref{eq:pb}) is said to follow the {\it correlative sparsity pattern} (csp) $(\mc{I}_1,\dots,\mc{I}_s)$ if
\begin{enumerate}[label=(\arabic*)]
\item there exist $f_1,f_2,\dots, f_s$ such that every $f_i\in\re[x^{(i)}]$ and $f(x)=f_1(x^{(1)})+\dots+ f_s(x^{(s)})$;
\item there exist partitions $I=(I_1,\dots, I_s)$ of $[m]$ and $E=(E_1,\dots, E_s)$ of $[\ell]$, such that for all $i\in[s]$,
we have $g_{j_1}\in \re[x^{(i)}]$ and $h_{j_2}\in \re[x^{(i)}]$ for every $j_1\in I_i$ and $j_2\in E_i$.
\end{enumerate}
For convenience, we let $m_i:=|I_i|$ and $\ell_i:=|E_i|$,
and denote 
\[g^{(i)}\,\coloneqq\,\left(g_j: j\in I_i \right),\quad
h^{(i)}\,\coloneqq\,\left( h_j: j\in E_i \right).\]
Then, both $g^{(i)}$ and $h^{(i)}$ are subsets of $\re[x^{(i)}]$, and the polynomial optimization~\eqref{eq:pb} with csp $(\mc{I}_1,\dots,\mc{I}_s)$ can be written in the following way: 
\begin{equation}\label{eq:cs-polyopt}\left\{
\begin{array}{lll}
&\min\limits_{x\in \R^n} & f_1(x^{(1)})+f_2(x^{(2)})+\dots+f_s(x^{(s)})\\
&\enspace \mathrm{s.t.} & g^{(1)}(x^{(1)})\ge 0 \ddd  g^{(s)}(x^{(s)})\ge 0,\\
& &  h^{(1)}(x^{(1)})= 0 \ddd  h^{(s)}(x^{(s)})= 0.
\end{array}\right.
\end{equation}

In this paper, we are interested in problems with csp $\{\cI_1,\ldots,\cI_s\}$ that satisfies the \textit{running intersection property} (RIP),
meaning that for each $1\leq i\leq s-1$, $\cI_{i+1} \cap \left( \cI_1 \cup \cdots \cup\cI_{i} \right)\subset \cI_t$ for some $ t\in \{1,\ldots,i\}$; see~\Cref{def:rip}.
The Moment-SOS relaxation with correlative sparsity is studied in \cite{waki2006sums},
and the convergence results are proved in \cite{GrimmNetzerSchweighofer07,Kojima09,Lasserre06,nie2009sparse} for the case when the RIP holds.
Recently, Magron and Wang developed the software {\tt TSSOS} \cite{magron2021tssos} that implements correlative and term sparse SOS relaxations for polynomial optimization (see also \cite{magron2023sparse,TSSOS,CSTSSOS}),
and it has been used in many applications \cite{newton2022sparse,wang2021certifying}.

Note that for any polynomial optimization problem,
the trivial csp, i.e., $s=1$ with $\mc{I}_1=[n]$, always exists.
Our primary interests lie in the cases where  $n$ is much bigger than $\max_{i\in[s]} |\mc{I}_i|$.
For polynomial optimization~\eqref{eq:pb} with the given csp,
we aim to construct reformulations similar to Nie's LME reformulation introduced in \cite{nie2019tight}, while maintaining the correlative sparsity of~\eqref{eq:pb}.
Our main contributions are as follows:
\begin{itemize}[leftmargin=2em]
\item For polynomial optimization with the given csp,
we provide a systematic way to construct correlatively sparse LMEs (CS-LMEs),
which are polynomial functions in $x$ and some auxiliary variables.
\item Based on CS-LMEs, we proposed correlatively sparse reformulations using the KKT optimality conditions.
We show that under some general conditions, the reformulation inherits the csp and the RIP from the original polynomial optimization,
and their optimal values are identical.
\item We show that for a given relaxation order, correlatively sparse SOS (CS-SOS) relaxations always provide tighter lower bounds for the optimal value of the polynomial optimization problem when the CS-LME reformulation is applied.
The asymptotic convergence of our approach is proved under some standard assumptions.
Numerical experiments are given to show the superiority of our CS-LME approach.
\end{itemize}

This paper is organized as follows.
Some preliminaries for polynomial optimization and Lagrange multiplier expressions are given in~\Cref{sc:pre}.
In~\Cref{sc:CSLME}, CS-LMEs are studied, and reformulations based on CS-LMEs are proposed.
\Cref{sc:CSSOS} studies the CS-SOS relaxations for solving CS-LME relaxations.
Numerical experiments are presented in~\Cref{sc:ne}, and we conclude our approach and discuss future work in~\Cref{sc:dis}. In~\Cref{sec:appLME}, we briefly recall the general methodology for the computation of LMEs and CS-LMEs.

\section{Preliminaries}
\label{sc:pre}
\subsection{Notation and definitions}
\label{subsec:nd}
Let $r$ be a positive integer. Denote $[r]:=\{1,\ldots,r\}$, and let $\bI_r$ be the $r$-by-$r$ identity matrix. 
When the dimension is clear, we use ${\bf 0}$ (resp., ${\bf 1}$) to denote the all-zero (resp., all-one) vector.
Given two vectors $v,w\in\re^r$, 
we denote by 
$v\circ w$ the entrywise product of $v$ and $w$, and $v\perp w $ means that $v^{\top}w=0$. For $v\in \R^r$ and  $1\le i\leq j\le r$, we denote by
$v_{i:j}$ the subvector formed by the elements of $v$ indexed from $i$ to $j$, i.e., $v_{i:j}:=\left [ v_{i} ,\cdots ,  v_{j}
\right ]^{\top}$.

Let $z=(z_1,\ldots,z_r)$ be a tuple of variables. Denote by $\R[z]$ the ring of polynomials in variables $z_1,\ldots,z_r$ with real coefficients,
and let $\R[z]^{r\times k}$ (resp., $\R[z]^{r}$) be the set of all $r\times k$ matrices (resp., $r$-dimensional vectors) whose entries are polynomials in $z$.  For a polynomial $p\in \R[z]$, denote by $\deg(p)$ the degree of $p$. 
For an integer $d\in \bN$, let $\R[z]_{d}$ be the $\R$-vector space of real polynomials in $r$ variables of degrees at most $d$. A polynomial $p\in \R[z]$ is a sum-of-squares (SOS) if there exist $\sigma_1, \ldots, \sigma_t \in \R[z]$ such that 
$p=\left(\sigma_1\right)^2+\dots+\left(\sigma_t\right)^2 $. 
Denote by $\Sigma[z]$ the set of SOS polynomials in $z$, and let $\Sigma[z]_d:=\Sigma[z] \cap \R[z]_{d}$.
For $p\in \R[z]$ and  $\mc{R},\mc{S} \subseteq \R[z]$, we define
$p \cdot \cR \,:=\, \{p\cdot q: q\in \cR\}$ and $ \mc{R}+\mc{S}\,\coloneqq\, \{r+s: r\in\mc{R},\, s\in\mc{S}\}.$

Given a tuple $g=(g_1,\ldots,g_{m})\subseteq\re[z]$, the \textit{quadratic module} of $\R[z]$ generated by $g$ is the set
\be\label{eq:qm}
\Qmod(g):= \Sigma[z]+g_1 \cdot \Sigma[z]+\dots+g_m \cdot \Sigma[z],
\ee
and the $2d$th truncation of $\Qmod(g)$ is the set
\be\label{eq:qmd}
\Qmod(g)_{2d}:= \Sigma[z]_{2d}+g_1 \cdot \Sigma[z]_{2d-\deg(g_1)}
+\dots +g_m \cdot \Sigma[z]_{2d-\deg(g_m)}.
\ee
For a tuple $h=(h_1,\ldots,h_{\ell})\subset \R[z]$, the \textit{ideal} of $\R[z]$ generated by $h$ is the set 
\[\Ideal(h):=h_1\cdot \R[z]+\dots+h_{\ell}\cdot \R[z],\]
and the $2d$th truncation of $\Ideal(h)$ is the set 
\[\Ideal(h)_{2d}:=h_1\cdot \R[z]_{2d-\deg(h_1)}+\dots+h_{\ell}\cdot \R[z]_{2d-\deg(h_{\ell})}.\]
For two polynomial tuples $h$ and $g$,
denote 
\be\label{eq:IQ}
\IQ(h,g)\,\coloneqq\, \Ideal(h)+\Qmod(g),
\quad \IQ(h,g)_{2d}\,\coloneqq\, \Ideal(h)_{2d}+\Qmod(g)_{2d}.\ee
Then, it is clear that every polynomial $p\in\IQ(h,g)\subseteq\re[z]$ is nonnegative over the set $\cK:=\{z\in \R^r: h(z)=0,\enspace g(z)\ge0\}$.
Conversely,
when $\IQ(h,g)$ is {\it archimedean}, i.e., when there exists $p\in \IQ(h,g)$ such that $\{z\in \R^r: p(z)\geq 0\}$ is compact (see \cite{lasserre2015introduction}),
all positive polynomials over $\cK$ are in $\IQ(h,g)$.
This result is referred to as Putinar's Positivstellensatz \cite{Putinar93}.
Moreover, when $h=0$ has finitely many real roots,
or when some general optimality conditions hold,
{a polynomial $f\in\re[z]$ is nonnegative over $\cK$ if and only if $f\in\IQ(h,g)_{2d}$ for all $d$ that is sufficiently large (see \cite{NieReal, Nie14}).}

Throughout this paper, $x=(x_1,\ldots,x_n)$ is the tuple of $n$ variables.
Given the csp $(\cI_1\ddd\cI_s)$,  for each $i\in[s]$, we fix a certain ordering for elements in $\cI_i$ and denote by $x^{(i)}$ the tuple of variables  $(x_k: k\in \cI_i)$.
The $j$th variable of $x^{(i)}$, denoted by $x^{(i)}_j$, corresponds to the variable $x_k$ if $j$ is the order of $k$ in  $\cI_i$. 
For example, if $\cI_1$ is ordered as $(1,3,5,6)$, then $x^{(1)}_2=x_3$.
For  polynomial $p\in \R[x]$, denote by $\nabla p\in \R[x]^n$ the gradient of $p$ and 
\be\label{eq:p-gradient}
\nabla_i p \deq \left [  \begin{array}{lll}\frac{\partial p}{\partial x^{(i)}_1} & \cdots & \frac{\partial p}{\partial x^{(i)}_{n_i}} \end{array}\right]^{\top}\in\re [x]^{n_i}.
\ee
When the dimension of the ambient space is clear,
we use $\ei_i$ to denote the $i$th standard basis vector whose $i$th entry is $1$ while all other entries are zeros.
For $k\in\cI_i$, denote 
\be\label{eq:eik} e^{(i)}_k\deq e_j\in\re^{n_i},\ee
where $j$ is the order of $k$ in the tuple $\cI_i$. 
For instance, if $\cI_1$ is ordered as $(1,3,5,6)$,  then $e^{(1)}_3 = e_2\in\re^4$.

\subsection{Moment-SOS relaxation}
\label{subsec:msr}
Denote by $f_{\min}$ the optimal value of the polynomial optimization problem~\eqref{eq:pb}.
Denote by $\mc{K}$ the feasible set of (\ref{eq:pb}), i.e., $\mc{K}:=\{x\in\re^n: h(x)=0,\ g(x)\ge0\}$.
Then finding the global minimum of (\ref{eq:pb}) is equivalent to
\begin{equation}\label{eq:Pcone}
\left\{\begin{array}{lll}
&\displaystyle\max ~~~ &\gamma\\
&\enspace \mathrm{s.t.} \quad  & f-\gamma\in \mc{P}_{d_0}(\mc{K}) .
\end{array}\right.
\end{equation}
In the above, $d_0$ is the degree of $f$, and $\mc{P}_{d_0}(\mc{K})$ is the cone of nonnegative polynomials over $\mc{K}$ with degrees not greater than $d_0$.
A computationally tractable relaxation for (\ref{eq:Pcone}) is called {\it the Moment-SOS relaxation}.
Given the relaxation order $d\in\mathbb{N}$ such that $2d\ge \max\{\deg(f), \deg(g), \deg(h)\}$,
the $d$th order SOS relaxation of (\ref{eq:Pcone}) (and~\eqref{eq:pb}) is
\begin{equation}\label{eq:msos}
\left\{\begin{array}{lll}
&\displaystyle\max ~~~ &\gamma\\
&\enspace \mathrm{s.t.} \quad  & f-\gamma \in  \IQ(h,g)_{2d}.
\end{array}\right.
\end{equation}
{Its dual problem corresponds to the so-called $d$th order moment relaxation of~\eqref{eq:pb},
and this primal-dual pair is referred to as the Moment-SOS relaxation.
Both \eqref{eq:msos} and its dual problems can be written as SDP problems.}
We refer to \cite{henrion2005detecting,Lasserre2001,lasserre2015introduction,lasserre2018moment,laurent2009sums, nie2014ATKM,nie2015linear,nie2023moment} for more references about polynomial optimization and moment problems.

For a relaxation order $d$,
denote by $\theta_d$ the optimal value of~\eqref{eq:msos}. Clearly $\theta_d$ provides a lower bound of $f_{\min}$, i.e. $\theta_d\leq f_{\min}$. 
Convergence of the Moment-SOS relaxation relies on Putinar's Positivestellenstaz~\cite{Putinar93}. 
\begin{theorem}[\cite{Lasserre2001}]
If $IQ(g,h)$ is archimedean, then
$
\lim_{d\rightarrow +\infty} \theta_d=f_{\min}.
$
\end{theorem}
We would like to remark that under some  conditions,
the Moment-SOS relaxations have finite convergence,
i.e., $\theta_d=f_{\min}$ for all $d$ that is big enough.
We refer to \cite{cifuentes2020geometry,de2011lasserre,hua2021exactness,NieReal,Nie14} for more related work.
The Moment-SOS relaxations have been implemented in the software {\tt GloptiPoly 3} \cite{henrion2009gloptipoly}.
In this paper, we also call Moment-SOS relaxations "dense relaxations" or "dense SOS relaxations" to distinguish them from SOS relaxations exploiting the sparsity.

\subsection{Correlatively sparse SOS relaxation}
\label{sc:cssos} Let us
consider the problem~\eqref{eq:cs-polyopt} with csp $(\cI_1,\ldots, \cI_s)$.
For polynomial tuples $h^{(i)},g^{(i)}\in\re[x^{(i)}]$, we
denote by $$\IQ_{\cI_i}(h^{(i)},g^{(i)})$$ the set given by (\ref{eq:qm})--(\ref{eq:IQ}) with $z=x^{(i)}$.
To exploit the correlative sparsity of problem~\eqref{eq:cs-polyopt}, we consider the following relaxation for problem~\eqref{eq:cs-polyopt}:
\begin{equation}\label{eq:msoscs}
\qquad
\left\{\begin{array}{lll}
&\displaystyle\max ~~~ &\gamma\\
&\enspace \mathrm{s.t.} \quad  & f-\gamma \in  \displaystyle \IQ_{\cI_1}\left( h^{(1)}, g^{(1)}\right)_{2d}+\dots+\IQ_{\cI_s}\left( h^{(s)}, g^{(s)}\right)_{2d}.
\end{array}\right.
\end{equation}
We refer to~\eqref{eq:msoscs} as the $d$th order {\it CS-SOS relaxation} of~\eqref{eq:cs-polyopt} \cite{Lasserre06,nie2009sparse,magron2023sparse,waki2006sums},
and denote its optimal value by $\rho_{d}$.
To demonstrate the convergence results for CS-SOS relaxations, we need the following property of csps.
\begin{definition}\label{def:rip}
We say that the csp $(\cI_1,\ldots, \cI_s)$ satisfies the RIP if for every $i\in [s-1]$, there exists $t\leq i$ such that
\begin{align}\label{a:rin}
\mc{I}_{i+1}\bigcap \bigcup_{j=1}^{i} \mc{I}_j \subseteq \mc{I}_t.
\end{align}
\end{definition}
Convergence of the CS-SOS relaxation is derived from the following sparse version of Putinar's Positivestellenstaz.
\begin{theorem}[\cite{GrimmNetzerSchweighofer07,Kojima09,Lasserre06}]\label{th:sparseputinar} Suppose 
$(\cI_1,\ldots,\cI_s)$ satisfies the RIP property, and $\IQ_{{\cI_i}}(g^{(i)},h^{(i)})$ is archimedean for each $i\in [s]$.
If $f(x)\deq f_1(x^{(1)})+\dots+f_s(x^{(s)})$ is  positive  on  the semialgebraic set $\bigcap_{i=1}^s\{x\in \R^n: g^{(i)}(x)\geq 0, h^{(i)}(x)=0\} $, then $$f\in \IQ_{\cI_1}\left( h^{(1)}, g^{(1)}\right)+\dots+\IQ_{\cI_s}\left( h^{(s)}, g^{(s)}\right).$$
\end{theorem}
Therefore, under the same conditions as that in~\Cref{th:sparseputinar}, we have
\begin{align}\label{a:rhodlmt}
\lim_{d\rightarrow +\infty} \rho_{d}= f_{\min}.
\end{align}

{Aside from the correlative sparsity,
one can also exploit the {\it term sparsity} of polynomial optimization problems,
or combine both kinds of sparsity to obtain  the so-called {\it correlative and term sparsity SOS relaxations} (CS-TSSOS) of~\eqref{eq:cs-polyopt},
whose convergence is guaranteed with the term sparsity being given by the {\it maximal chordal extension} when the CS-SOS relaxation is convergent \cite{CSTSSOS}.
Since this paper mainly concerns  correlative sparsity,
 we  refer to \cite{magron2023sparse,TSSOS,CSTSSOS} for more details on the exploitation of term sparsity.
The CS-TSSOS relaxations have been recently implemented in the  software {\tt TSSOS} \cite{magron2021tssos}.

\subsection{Optimality conditions and Lagrange multiplier expressions}
\label{sc:preLMEs}
For the polynomial optimization problem \eqref{eq:pb},
the KKT conditions can be described by the following polynomial system in $(x,\lambda)\in\R^{n+m+\ell}$:
\begin{align}\label{a:KKT}
\left\{ 
\begin{array}{c}
\nabla f(x)=\displaystyle\sum_{j=1}^m \lambda_j \nabla g_j(x)+\sum_{j=1}^{\ell}\lambda_{m+j} \nabla h_j(x),\\
h(x)=0,\ 
0\le \lmd_{1:m}\perp g(x)\ge 0.
\end{array}\right.
\end{align}
The pair $(x,\lmd)$ satisfying \eqref{a:KKT} is called a KKT pair,
and the first component $x$ of a KKT pair is called a KKT point of~\eqref{eq:pb}.
Under some constraint qualification conditions,
every minimizer of (\ref{eq:pb}), if it exists, must be a KKT point.
In this case,
minimizing $f$ over the KKT system~\eqref{a:KKT} returns the same optimal value and optimal solutions as the original problem~\eqref{eq:pb}. 
Moreover, conditions guaranteeing the convergence of the dense SOS relaxations are milder for the minimization over the KKT ideal than that for the original problem~\eqref{eq:pb}. 
In particular, convergence can still occur even when the semialgebraic set given by \eqref{a:KKT} is noncompact~\cite{Nie07JPAA,NieDemmelSturmfels06}. 

A drawback, however, of working on the KKT system~\eqref{a:KKT} rather than on the original feasible region $\cK\subseteq \R^n$ is the augmentation of the number of variables from $n$ to $n+m+\ell$,
which causes a significant increase on the computational cost.
To deal with this undesired complexity growth,
Nie~\cite{nie2019tight} proposed polynomial Lagrange multipliers expressions.
For the polynomial optimization problem (\ref{eq:pb}),
let $\hm:=m+\ell$,
and let $c:=(c_1,\dots,c_{\hat m})$ be an enumeration for the constraining pair $(g,h)$.
We denote 
\begin{equation}\label{eq:gsedf}
\bG(x):=\left [ \begin{array}{cccccc}
\nabla c_1(x) & \nabla c_2(x)  & \cdots & \nabla c_{\hm}(x) \\
c_1(x) & 0  & \cdots & 0 \\
0 & c_2(x)  & \cdots & 0 \\
\vdots  & \vdots & \ddots & \vdots \\
0  & \cdots & 0 & c_{\hm}(x) \\
\end{array} \right ], \quad 
\mathbf{f}(x):= \left [\begin{array}{c}
\nabla f(x) \\ 0 \\ \vdots \\0
\end{array}
\right ].
\end{equation}
Then, the following equation holds at every KKT pair $(x,\lmd)$:
\begin{align}\label{a:GMatrix}
\bG(x) \cdot \lambda=\mathbf{f}(x).
\end{align}
If there exists a matrix of polynomials $L\in \R[x]^{\hm\times n}$ and $ D \in \R[x]^{\hm\times \hm}$  such that
\begin{align}\label{a:LGI}
\left [\begin{array}{cc}L(x) & D(x)
\end{array}
\right ]\bG(x)=\bI_{\hm},\enspace \forall x\in \R^n,
\end{align}
then the Lagrange multipliers $\lambda$ can be  expressed as polynomials in $x$:
\begin{equation}\label{eq:lme1}
\left [\begin{array}{c} \lambda_1 \\ \vdots \\ 
\lambda_{\hm}
\end{array}
\right ]
=  
\left [\begin{array}{c}
p_1(x)\\ \vdots \\ p_{\hm}(x)
\end{array}
\right ]:= L(x)\nabla f(x).
\end{equation}
The polynomial vector $p(x):=(p_1(x)\ddd p_{\hm}(x))$ is called the {\it Lagrange multiplier expression} (LME).
Denote
\[
c_{eq}(x):=\left [ \begin{array}{c} \nabla f(x)-\displaystyle\sum_{j=1}^m p_j(x)\nabla g_j(x)-\sum_{j=1}^{\ell}p_{m+j}(x)\nabla h_j(x) \\
h(x)\\
p_{1:m}(x) \circ g(x)
\end{array}\right],
\]
and
\[c_{in}(x):=\left[
\begin{array}{c}
p_{1:m}(x)  \\
g(x) 
\end{array}
\right].\]
Then, $x\in\re^n$ is a KKT point if and only if $x$ satisfies $c_{eq}(x)=0,\ c_{in}(x)\ge0.$
Based on the LME \eqref{eq:lme1}, Nie~\cite{nie2019tight} proposed the following reformulation of (\ref{eq:pb}): 
\begin{equation}\label{eq:pblme}
\left\{\begin{array}{lll}
&\displaystyle\min_{x\in \R^{n}} ~~~ &f(x)\\
&\enspace \mathrm{s.t.} \quad  &
c_{eq}(x)=0,\quad  c_{in}(x)\geq 0
\end{array}\right.
\end{equation}
It is clear that when the minimum of (\ref{eq:pb}) is attained at some KKT points,
the optimal values of (\ref{eq:pb}) and (\ref{eq:pblme}) are identical. 
In fact, the existence of LMEs guarantees that every minimizer of (\ref{eq:pb}), if it exists, must be a KKT point \cite[Proposition~5.1]{nie2019tight},
thus solving (\ref{eq:pb}) is equivalent to solving the reformulation (\ref{eq:pblme}).
When Moment-SOS relaxations are applied,
finite convergence is guaranteed under some generic conditions:
\begin{theorem}\cite[Theorem~3.3]{nie2019tight}
Suppose LMEs exist and (\ref{eq:pblme}) has a nonempty feasible set.
Denote by $\theta_d$ the optimal value of the $d$th order SOS relaxation~\eqref{eq:msos} of   the polynomial optimization problem~\eqref{eq:pblme}.
Then, we have that $f_{\min}=\theta_d$ holds for all $d$ big enough if $IQ(c_{eq},c_{in})$ is archimedean and the minimum value of (\ref{eq:pb}) is attained at a KKT point.
\end{theorem}

Recently, LMEs have been widely used in various problems given by polynomial functions,
such as bilevel polynomial optimization, Nash equilibrium problems, tensor computation, etc.
We refer to \cite{nie2020nash,NieTang21,nie2021lagrange,nie2018complete,nie2021saddle} for applications of LMEs.

{One wonders when  LMEs exist, i.e., when there exist matrices $L(x), D(x)$ such that~\eqref{a:LGI} holds.
We say that the constraining tuple $(g,h)$ is {\it nonsingular} if the matrix $\bG(x)$ given in (\ref{eq:gsedf}) has a full column rank for all $x\in\mathbb{C}^n$.
For (\ref{eq:pb}), LMEs exist if and only if its constraining tuple is nonsingular \cite[Proposition~5.1]{nie2019tight}.
We would like to remark that when the polynomials $c_1\ddd c_{\hm}$ are generic\footnote{We say a property holds generically if it holds for all points of input data but a set of Lebesgue measure zero.},
the nonsingularity condition holds.
However, there are cases when LMEs do not exist; see~\cref {ep:ball} for a concrete example and also~\cite{nie2019tight,NieTang21} for more details. 
In the following example, we give the matrices $L(x)$ and $D(x)$ for a special box-constrained problem.  The general methodology for formulating  LMEs
can be found in~\Cref{sec:appLME}.
\begin{example}\label{ep:box}
Consider the polynomial optimization problem with box constraints
\begin{equation}\label{eq:box}
\left\{\begin{array}{lll}
\displaystyle\min_{x\in \R^4} & f(x_1,x_2,x_3,x_4):=x_1^4x_2^2+x_1^2x_2^4+x_3^6\\
& \qquad\qquad\qquad\qquad\qquad -3x_1^2x_2^2x_3^2+x_3^3+x_3x_4^2-2x_3^2x_4\\
\enspace \mathrm{s.t.}  & x_1 \ge 0,\ 1-x_1\ge0,\ x_2\ge0,\ 1-x_2\ge0,\\
& x_3 \ge 0,\ 1-x_3\ge0,\ x_4\ge0,\ 1-x_4\ge0.
\end{array}\right.
\end{equation}
Note that in this problem, since all variables are nonnegative,
by the inequality of arithmetic and geometric means,
we have
\[\begin{aligned}
x_1^4x_2^2+x_1^2x_2^4+x_3^6 &\ge 3\sqrt[3]{x_1^4x_2^2\cdot x_1^2x_2^4\cdot x_3^6}=3x_1^2x_2^2x_3^2,\\
x_3^3+x_3x_4^2&\ge 2\sqrt{x_3^3\cdot x_3x_4^2} =2x_3^2x_4,
\end{aligned}
\]
where the equailities hold when $x_1=x_2=\dots=x_4$.
So the global minimum of \eqref{eq:box} is $0$ with minimizers $(t,t,t,t)$ for all $t\in[0,1]$.
Let $g(x):=(g_1(x)\ddd g_{8}(x))$ with
\be\label{eq:boxconstraints}
\begin{array}{l}
g_1(x) = x_1,\quad g_2(x) = 1-x_1, \quad g_3(x) = x_2, \quad g_4(x) = 1-x_2,\\
g_5(x) = x_3,\quad g_6(x) = 1-x_3, \quad g_7(x) = x_4, \quad g_8(x) = 1-x_4.
\end{array}
\ee
 The constraining tuple $g$ is nonsingular and~\eqref{a:LGI} holds with  
$$L(x)=\mbox{diag}(L_1(x),L_2(x),L_3(x),L_4(x)),\enspace D(x)=\mbox{diag}(D_1(x),D_2(x),D_3(x),D_4(x)) $$
being block-diagonal matrices. The matrices in the diagonal of $L$ are given by 
 $L_i(x)=\left[\begin{array}{l} 1-x_i\\ -x_i\end{array} \right]$ and  the matrices in the diagonal of $D$ are given by $D_i(x)=\left[\begin{array}{ll}1 & 1\\1 & 1
 \end{array}\right]$ for each $i\in [4]$.
Accordingly, the LMEs are
\be\label{eq:boxlme}
p_{2i-1}(x) = (1-x_i)\cdot \frac{\pt f}{\pt x_{i}}(x),\quad p_{2i}(x) = -x_i\cdot \frac{\pt f}{\pt x_i}(x),\quad (i=1\ddd 4).
\ee
In particular,
 $p_5(x),\, p_6(x)$  can be explicitly written as
\be \label{eq:p5p6}
\begin{array}{l}
p_5(x) = 6(x_1^2x_2^2x_3^2 - x_3^6 - x_1^2x_2^2x_3 + x_3^5)\\
 \qquad\qquad\qquad - 3x_3^3 + 4x_3^2x_4 - x_3x_4^2 + 3x_3^2 - 4x_3x_4 + x_4^2,\\
p_6(x) =6x_1^2x_2^2x_3^2 - 6x_3^6 - 3x_3^3 + 4x_3^2x_4 - x_3x_4^2,
\end{array}
\ee
which involve of the four variables $x_1, x_2, x_3, x_4$.
\end{example}
}

\section{Correlatively sparse LMEs and reformulations}
\label{sc:CSLME}
We consider the polynomial optimization problem~\eqref{eq:cs-polyopt} with the csp $(\cI_1\ddd \cI_s)$  satisfying the RIP. 
For $i\in [s]$,
we denote by $\bG^{(i)}$ the polynomial matrix $\bG$ given in~\eqref{eq:gsedf} associated with $(g^{(i)},h^{(i)})$. That is, if we let $c^{(i)}=(g^{(i)},h^{(i)})$, and $\hat{m}_i:=m_i+\ell_i$, then
\begin{equation}\label{eq:gisedf}
\bG^{(i)}(x^{(i)}):=\left [ \begin{array}{cccccc}
\nabla_i c^{(i)}_1(x^{(i)}) & \nabla_i c^{(i)}_2(x^{(i)})  & \cdots & \nabla_i c^{(i)}_{\hat{m}_i}(x^{(i)}) \\
c^{(i)}_1(x^{(i)}) & 0  & \cdots & 0 \\
0 & c^{(i)}_2(x^{(i)})  & \cdots & 0 \\
\vdots  & \vdots & \ddots & \vdots \\
0  & \cdots & 0 & c^{(i)}_{\hat{m}_i}(x^{(i)}) \\
\end{array} \right ].
\end{equation}

As mentioned in Section~\ref{sc:preLMEs},
one can reformulate polynomial optimization problems with LMEs,
from which the Moment-SOS relaxation gives a tighter lower bound for the polynomial optimization.
To apply LMEs,
the KKT system of~\eqref{eq:cs-polyopt} corresponds to the following semialgebraic set on $x\in \R^n$,  $\lambda^{(1)}\in \R^{\hat{m}_1},\ldots,\lambda^{(s)}\in \R^{\hat{m}_s}$:
\be\label{eq:cs-KKT}
\left\{\begin{array}{c}\nabla f_1 (x)+\dots+\nabla f_s (x)  = \displaystyle \sum_{i=1}^s \left(\sum_{j=1}^{m_i}\lmd^{(i)}_{j} \nabla g^{(i)}_{j}(x)+ \sum_{j=1}^{\ell_i}\lambda^{(i)}_{m_i+j} \nabla h^{(i)}_{j}(x)
\right),\\
h^{(i)}(x)=0,\enspace i\in [s],\\
0\le \lmd_{1:m_i}^{(i)}\perp g^{(i)}(x)\ge 0,\enspace i\in [s].
\end{array}
\right.
\ee
Hereafter,
we  additionally assume that the nonsingularity condition holds for the constraining pair $(g^{(i)},h^{(i)})$ within every $\cI_i$.
That is:
\begin{assumption}\label{ass:lme}
For each $i\in [s]$, there exist polynomial matrices $L^{(i)}(x^{(i)}) \in \R[x^{(i)}]^{\hat{m}_i\times n_i}$ and 
$D^{(i)}(x^{(i)}) \in \R[x^{(i)}]^{\hat{m}_i \times \hat{m}_i}$ such that
\begin{equation}
\label{eq:cs-LGI}
\left[ 
\begin{array}{ll}
L^{(i)}(x^{(i)}) & D^{(i)}(x^{(i)})
\end{array}\right] \bG^{(i)}(x^{(i)})
=\bI_{\hat{m}_i}.
\end{equation}
\end{assumption}
By \cite[Proposition~5.2]{nie2019tight},
(\ref{eq:cs-LGI}) holds if and only if the matrix $\bG^{(i)}(x^{(i)})$ have a full column rank for all $x^{(i)}\in\mathbb{C}^{n_i}$.
For such cases, we say the pair $(g^{(i)},h^{(i)})$ is {\it nonsingular}.
This is satisfied if all polynomials in $g^{(i)}$ and $h^{(i)}$ are generic polynomials in $x^{(i)}$.

{\subsection{Limitation of the original LME for exploiting correlative sparsity}
For the polynomial optimization \eqref{eq:cs-polyopt} with correlative sparsity, LMEs exist if and only if the constraining tuple of all the constraints is nonsingular, 
by \cite[Proposition~5.1]{nie2019tight}.
In general, 
\Cref{ass:lme} is a necessary but not sufficient condition of the nonsingularity for the constraining tuple $(g,h)$ of all constraints in~\eqref{eq:cs-polyopt}. This can be seen in the following example.} 
{
\begin{example}\label{ep:ball}
Consider the following polynomial optimization problem with three variables $x=(x_1, x_2,x_3)$ and two constraints 
\begin{equation}\label{eq:ex11}
\left\{\begin{array}{lll}
&\displaystyle\min_{x\in\re^3} ~~~ &f_1(x_1,x_2)+f_2(x_2,x_3)\\
&\enspace \mathrm{s.t.} \quad  &
1-x_1^2 -x_2^2\geq 0\\
&& 1-x_2^2- x_3^2 \geq 0 
\end{array}\right.
\end{equation}
Let $\cI_1=\{1,2\}$ and $\cI_2=\{2,3\}$.
Then~\eqref{eq:ex11} has the csp $(\cI_1, \cI_2)$ with
\be\label{eq:ball-csp}
g^{(1)}=\left(1-x_1^2 -x_2^2\right),\quad g^{(2)}=\left(1-x_2^2- x_3^2\right),\quad h^{(1)}=h^{(2)}=\emptyset.\ee
 The matrix $\bG(x)$ associated to~\eqref{eq:ex11} is
$$
\bG(x)=\left [ \begin{array}{cc}
-2x_1  & 0 \\
-2x_2  & -2x_2\\
0 & -2x_3  \\
1-x_1^2- x_2^2 & 0 \\
0 & 1-x_2^2 - x_3^2
\end{array} \right ],
$$
whose rank is $1$ at $x=(0,1,0)$.
Thus the constraining tuple of (\ref{eq:ex11}) is not nonsingular,
and LMEs do not exist.
On the other hand,
we have
$$\bG^{(1)}(x_1,x_2)=\left[ \begin{array}{c}
-2x_1 \\ -2 x_2 \\ 1-x_1^2-x_2^2
\end{array}\right],\enspace \bG^{(2)}(x_2,x_3)=\left[ \begin{array}{c}
-2x_2 \\ -2 x_3 \\ 1-x_2^2-x_3^2
\end{array}\right]. $$
One may check  that \Cref{ass:lme} holds with
\be \label{eq:L1L2}L^{(1)}(x_1,x_2)= \left [-\frac{1}{2}x_1 \enspace \enspace -\frac{1}{2}x_2\right ],\enspace L^{(2)}(x_2,x_3)= \left [-\frac{1}{2}x_2 \enspace \enspace -\frac{1}{2}x_3\right ],\ee
and $D^{(1)}=D^{(2)}=1$.
\end{example}
}
\begin{remark}
See also~\Cref{ep:box-ne}(ii), \Cref{ex:5balls}, \Cref{ep:nonaive} and \Cref{ep:ls} in~\Cref{sc:ne} for cases which satisfy~\Cref{ass:lme} but do not admit LMEs.
\end{remark}

{Another concern related to the original LME approach is that the LME reformulation (\ref{eq:pblme}), if exists, usually cannot inherit the csp of (\ref{eq:cs-polyopt}).
Indeed, the LME reformulation (\ref{eq:pblme}) may have constraints that involve all the variables, as demonstrated by the following example.
\begin{example}\label{ep:boxexplain}
Consider the polynomial optimization problem~\eqref{eq:box} with box constraints.
Let $\mc{I}_1=\{1,2,3\}$ and $\cI_2=\{3,4\}$.
Then (\ref{eq:box}) has the csp $(\cI_1,\cI_2)$ with $h^{(1)}=h^{(2)}=\emptyset$ and
\be\label{eq:box-csp}
\begin{aligned}
&f_1(x^{(1)}) = x_1^4x_2^2+x_1^2x_2^4+x_3^6-3x_1^2x_2^2x_3^2, \quad
f_2(x^{(2)}) =  x_3^3+x_3x_4^2-2x_3^2x_4,\\
&g^{(1)}(x^{(1)}) = (x_1,\, 1-x_1,\, x_2,\, 1-x_2), \quad
g^{(2)}(x^{(2)}) = (x_3,\, 1-x_3,\, x_4,\, 1-x_4).
\end{aligned}\ee
In view of~\eqref{eq:p5p6}, the LME reformulation~\eqref{eq:pblme} of~\eqref{eq:box} does not have correlative sparsity, as the nonnegativity conditions for Lagrange multipliers $p_5(x)\ge0$ and $p_6(x)\ge0$ involve all variables.
\end{example}}
In the next two subsections,
we provide a systematic method to construct LMEs for~\eqref{eq:cs-polyopt} which leverages the correlative sparsity pattern $(\mc{I}_1,\dots, \mc{I}_s)$.

\subsection{Correlatively sparse LMEs: two blocks}\label{subsec:tb}

We begin with the case of two blocks, i.e.,  $s=2$.
{Before giving a formal presentation of our approach, we would like to expose the underlying idea through the following example of three variables:
\begin{equation}\label{eq:ex1g}
\left\{\begin{array}{lll}
&\displaystyle\min ~~~ &f_1(x_1,x_2)+f_2(x_2,x_3)\\
&\enspace \mathrm{s.t.} \quad  &
g_1(x_1,x_2)\geq 0,\\
&& g_2(x_2,x_3) \geq 0.
\end{array}\right.
\end{equation}
The problem~\eqref{eq:ex1g} has the csp $(\cI_1,\cI_2)$ with $\cI_1=\{1,2\}$ and $\cI_2=\{2,3\}$.
Recall that for each $i\in [s]$, the partial gradient $\nabla_i$ is defined as in~\eqref{eq:p-gradient}.
Under \Cref{ass:lme}, there exist polynomial matrices $L^{(1)}\in \R[x_1,x_2]^2,\, D^{(1)}\in \R[x_1,x_2],\, L^{(2)}\in \R[x_2,x_3]^2$, and $D^{(2)}\in \R[x_2,x_3]$ such that for each $i=1,2$,
\be 
\label{eq:exllme}
\left[ \begin{array}{ll} L^{(i)}(x_1,x_2) & D^{(i)}(x_1,x_2)\end{array}\right] \left[ \begin{array}{c}
\nabla_i g_i(x_1,x_2)\\
g_i(x_1,x_2)
\end{array}\right]=1.
\ee
The KKT system of~\eqref{eq:ex1g} is
\be\label{eq:cs-KKTex11}
\left\{\quad \begin{aligned}
\frac{\partial f_1}{ \partial x_1} (x_1,x_2)\ &= \  \lambda_1\cdot \frac{ \partial g_1 }{\partial x_1}(x_1,x_2),\\
\frac{\partial f_1}{\partial x_2}(x_1,x_2)+\frac{\partial f_2}{\partial x_2}(x_2,x_3)\ &=\ \lambda_1\cdot \frac{\partial g_1} {\partial x_2} (x_1,x_2)+\lambda_2\cdot \frac{ \partial g_2 }{\partial x_2}(x_2,x_3), \\
\frac{\partial f_2}{\partial x_3}(x_2,x_3)\ 
&= \ 
\lambda_2\cdot \frac{\partial g_2} {\partial x_3} (x_2,x_3) 
,\\
0\ \le\  g_1(x_1,x_2)\ &\perp \ \lmd_{1}\ \ge\ 0, \\
0\ \le\  g_2(x_2,x_3)\ &\perp \ \lmd_{2}\ \ge\ 0.
\end{aligned}
\right.
\ee
Clearly, the csp structure is broken when $f_1,f_2,g_1,g_2$ are dense polynomials,
due to the second equation above.
Introducing an auxiliary variable $\nu$, we rewrite~\eqref{eq:cs-KKTex11} as
\be\label{eq:cs-KKTex11r}
\left\{\quad \begin{aligned}
\nabla_1 f_1(x_1,x_2) + \begin{bmatrix} 0 \\ \nu\end{bmatrix}\  & =\  \lambda_1\cdot \nabla_1 g_1 (x_1,x_2),\\
0\le \  g_1(x_1,x_2)\ &\perp \ \lmd_1\ \ge 0, \\
\nabla_2 f_2(x_2,x_3) - \begin{bmatrix} \nu \\ 0\end{bmatrix}\  & =\  \lambda_2\cdot \nabla_2 g_2 (x_2,x_3),\\
0\le \  g_2(x_2,x_3)\ &\perp \ \lmd_2\ \ge 0, 
\end{aligned}
\right.
\ee
Thus by~\eqref{eq:exllme}, for any $(x_1,x_2,\lambda_1,\lambda_2, \nu)$ satisfying~\eqref{eq:cs-KKTex11r}, we must have
\be \label{eq:lambda1lambda2}
\begin{aligned}
\lambda_1 &=L^{(1)}(x_1,x_2) \left(\nabla_1 f_1(x_1,x_2) + \begin{bmatrix} 0 \\ \nu\end{bmatrix}\right),\\
\lambda_2 &=L^{(2)}(x_2,x_3) \left(\nabla_2 f_2(x_2,x_3) - \begin{bmatrix} \nu \\ 0\end{bmatrix}\right).
\end{aligned}
\ee
Under some constraint qualification conditions,
we arrive at a reformulation for (\ref{eq:ex1g}) which possess the  csp with two blocks of variables $$(x_1,x_2,\nu), \enspace (x_2,x_3,\nu)$$
by plugging~\eqref{eq:lambda1lambda2} back into~\eqref{eq:cs-KKTex11r} to replace $\lambda_1$ and $\lambda_2$.

\begin{example}\label{ep:ball-cslme}
Consider the polynomial optimization problem \eqref{eq:ex11} as a special case of~\eqref{eq:ex1g}.  
Recall that \Cref{ass:lme} holds with $L^{(1)}$ and $L^{(2)}$ given in~\eqref{eq:L1L2}.
In view of~\eqref{eq:lambda1lambda2}, we have
\be\label{eq:ball-cslme}
\begin{aligned}
\lambda_1= p^{(1)}(x_1,x_2,\nu)& :=
-\frac{x_1}{2} \frac{\partial f_1}{\partial x_1} (x_1,x_2)-\frac{x_2}{2} \frac{\partial f_1}{\partial x_2} (x_1,x_2)-\frac{x_2}{2}  \nu, \\
\lambda_2=p^{(2)}(x_2,x_3,\nu)& :=-\frac{x_2}{2} \frac{\partial f_2}{\partial x_2} (x_2,x_3) -\frac{x_3}{2}  
\frac{\partial f_2}{\partial x_3} (x_2,x_3)+\frac{x_2}{2} \nu.
\end{aligned}\ee
Suppose the minimum value of (\ref{eq:ex11}) is attained at a KKT point $x^*$.
Then there exists $\nu^*\in\re$ such that (\ref{eq:cs-KKTex11r}) holds at $(x^*,\nu^*)$ with $\lmd_1,\lmd_2$ given by (\ref{eq:ball-cslme}).
Taking (\ref{eq:cs-KKTex11r}) as constraints with $\lmd_i$ being substituted by $p^{(i)}$ for every $i=1,2$,
we arrive at the following optimization problem:
\be\label{eq:ball-cslme-reform}
\left\{\begin{array}{lll}
&\displaystyle\min_{x,\nu} ~~~ &f_1(x_1,x_2)+f_2(x_2,x_3)\\
&\enspace \mathrm{s.t.} \quad  &
\nabla_1 f_1(x_1,x_2) + \begin{bmatrix} 0 \\ \nu\end{bmatrix} = -2p^{(1)}(x_1,x_2,\nu)\cdot \begin{bmatrix} x_1 \\ x_2\end{bmatrix},\\[11pt]
&& 0\le \left(1-x_1^2-x_2^2\right)\perp p^{(1)}(x_1,x_2,\nu) \ge0,\\[5pt]
&& \nabla_2 f_2(x_2,x_3) - \begin{bmatrix} \nu \\ 0\end{bmatrix} = -2p^{(2)}(x_2,x_3,\nu)\cdot \begin{bmatrix} x_2 \\ x_3\end{bmatrix},\\[11pt]
&& 0\le \left(1-x_2^2-x_3^2\right)\perp p^{(2)}(x_2,x_3,\nu) \ge0.
\end{array}\right.
\ee
Then $(x^*,\nu^*)$ is a global minimizer for (\ref{eq:ball-cslme-reform}).
As we will formally introduce later, polynomials $p^{(1)},p^{(2)}$ representing $\lmd_1,\lmd_2$ are called CS-LMEs, and (\ref{eq:ball-cslme-reform}) is called the {\it CS-LME reformulation} for (\ref{eq:ex11}).

Recall from~\Cref{ep:ball} that~\eqref{eq:ex11} does not admit LMEs, thus the LME reformulation (\ref{eq:pblme}) is not available for (\ref{eq:ex11}). One may consider a reformulation of~\eqref{eq:ex11} using the KKT system~\eqref{eq:cs-KKTex11} by taking $\lmd_1,\lmd_2$ as new variables. Then the total number of variables in this approach is $5$ and there is no correlative sparsity anymore. Instead, by appropriately adding extra variable $\nu$, we obtained the CS-LME reformulation (\ref{eq:ball-cslme-reform}) which maintains to a  degree the original csp structure: we have three variables in each of the two blocks. 
\end{example}
}

Now we present formally the CS-LME approach for the polynomial optimization problem~\eqref{eq:cs-polyopt} with two block csp structure.
{Given the csp $(\cI_1,\cI_2)$, we introduce extra variables $\nu:=(\nu_k)_{k\in\cI_1\cap\cI_2}$.
Then, the gradient of the objective function $\nabla f_1(x)+\nabla f_2(x)$ can be split into two terms such that one only involves $(x^{(1)},\nu)$ and the other one only has $(x^{(2)},\nu)$.
Recall that for  $i\in \{1,2\}$ and $k\in\cI_i$, the vector $\ei^{(i)}_k$ is defined in (\ref{eq:eik}). 
Let 
\be \label{eq:F12}
\begin{array}{l}F^{(1)}(x^{(1)},\nu)\deq \nabla_1 f_1(x^{(1)})+\sum_{k\in\cI_1\cap\cI_2}\nu_k\ei^{(1)}_k\in\re^{n_1}, \\
F^{(2)}(x^{(2)},\nu)\deq \nabla_2 f_2(x^{(2)})-\sum_{k\in\cI_1\cap\cI_2}\nu_k\ei^{(2)}_k\in\re^{n_2}.
\end{array}
\ee
Then,
$(x,\lmd^{(1)},\lmd^{(2)})$ is a KKT tuple of (\ref{eq:cs-polyopt}) if and only if there exists $\nu=(\nu_k)_{k\in\cI_1\cap\cI_2}$ such that}
\be\label{eq:cs-KKT-2}
\left\{\begin{array}{c}
F^{(1)}(x^{(1)},\nu)  = \displaystyle \sum\nolimits_{j=1}^{m_1}\lmd^{(1)}_{j} \nabla_1 g^{(1)}_{j}(x^{(1)})+ \sum\nolimits_{j=1}^{\ell_1}\lambda^{(1)}_{m_1+j} \nabla_1 h^{(1)}_{j}(x^{(1)}),\\[6pt]
F^{(2)}(x^{(2)},\nu)  = \displaystyle \sum\nolimits_{j=1}^{m_2}\lmd^{(2)}_{j} \nabla_2 g^{(2)}_{j}(x^{(2)})+ \sum\nolimits_{j=1}^{\ell_2}\lambda^{(2)}_{m_2+j} \nabla_2 h^{(2)}_{j}(x^{(2)}),\\[7pt]
h^{(1)}(x^{(1)})=0, \quad h^{(2)}(x^{(2)})=0,\\[2pt]
0\le \lmd_{1:m_1}^{(1)}\perp g^{(1)}(x^{(1)})\ge 0,\quad 
0\le \lmd_{1:m_2}^{(2)}\perp g^{(2)}(x^{(2)})\ge 0.
\end{array}
\right.
\ee
Under Assumption~\ref{ass:lme},
if we let
\[p^{(1)}(x^{(1)},\nu)\deq L^{(1)}(x^{(1)})F^{(1)}(x^{(1)},\nu),\quad 
p^{(2)}(x^{(2)},\nu)\deq L^{(2)}(x^{(2)})F^{(2)}(x^{(2)},\nu),\]
then by (\ref{eq:cs-LGI}), for any $(x,\lmd^{(1)},\lmd^{(2)},\nu)$ satisfying (\ref{eq:cs-KKT-2}), we have
\[\lmd^{(1)} = p^{(1)}(x^{(1)},\nu),\quad
\lmd^{(2)} =p^{(2)}(x^{(2)},\nu).\]
The polynomial vectors $p^{(1)},\, p^{(2)}$ are called {\it correlatively sparse Lagrange multiplier expression} (CS-LME) for $\lmd^{(1)}$ and $\lmd^{(2)}$, respectively.
Replacing $\lambda^{(1)},\,\lambda^{(2)}$ by the polynomial vectors $p^{(1)}(x^{(1)},\nu)$ and $p^{(2)}(x^{(2)},\nu)$, we get the following  reformulation of~\eqref{eq:cs-polyopt}:
\begin{equation}\label{eq:pblme-2}
\left\{\begin{array}{lll}
\displaystyle\min_{x\in \R^{n}} &f_1(x^{(1)})+f_2(x^{(2)})\\
\enspace \mathrm{s.t.}  &  \displaystyle
F^{(i)}(x^{(i)},\nu) = \sum\nolimits_{j=1}^{m_i}p^{(i)}_{j}(x^{(i)},\nu) \nabla_i g^{(i)}_{j}(x^{(i)})
\\&\qquad \qquad \qquad\enspace 
\displaystyle+\sum\nolimits_{j=1}^{\ell_i}p^{(i)}_{m_i+j}(x^{(i)},\nu) \nabla_i h^{(i)}_{j}(x^{(i)}), & (i=1,2)\\[6pt]
& h^{(i)}(x^{(i)})=0, \quad 0\le p_{1:m_i}^{(i)}(x^{(i)},\nu)\perp g^{(i)}(x^{(i)})\ge 0. & (i=1,2)
\end{array}\right.
\end{equation}
The reformulation~\eqref{eq:ball-cslme-reform} in~\Cref{ep:ball-cslme} is a special case of~\eqref{eq:pblme-2}.
One may check the polynomial optimization problem~\eqref{eq:pblme-2} has the csp with two blocks of variables:
\[ (x^{(1)},\nu),\quad
(x^{(2)},\nu).\]
{\begin{example}\label{ep:box-cslme}
Consider the polynomial optimization problem with box constraints (\ref{eq:box}) in~\Cref{ep:box}.
Its csp is given in~\Cref{ep:boxexplain},
and we have $\cI_1\cap\cI_2=\{3\}$. 
So we need to introduce a new variable $\nu\in\re$.
The $f_1,f_2,g^{(1)},g^{(2)}$ are given as in (\ref{eq:box-csp}), and we let
\[F^{(1)}(x^{(1)},\nu)= \nabla_1 f_1(x^{(1)}) + \nu\ei_3^{(1)},\quad
F^{(2)}(x^{(2)},\nu)=\nabla_2 f_2(x^{(2)}) - \nu\ei_3^{(2)}.\]
Moreover, denoting by $F^{(i)}_{j}$ the $j$th entry of $F^{(i)}$ for $j\in \{1,2\}$, we get CS-LMEs:
\be\begin{array}{lll}
\label{eq:box-cslme}
p^{(1)}_{2j-1}(x^{(1)},\nu)= (1-x_j)F^{(1)}_{j}(x^{(1)},\nu),  & p^{(1)}_{2j}(x^{(1)},\nu)=-x_jF^{(1)}_{j}(x^{(1)},\nu),\\[3pt]
p^{(2)}_{2j-1}(x^{(2)},\nu)=(1-x_{2+j})F^{(2)}_{j}(x^{(2)},\nu), & p^{(2)}_{2j}(x^{(2)},\nu)=-x_{2+j}F^{(2)}_{j}(x^{(2)},\nu).&\\
\end{array}
\ee
Note that when CS-LMEs are given as above,
the first equality constraints in (\ref{eq:pblme-2}) are reduced to
one single equation $F^{(1)}_{3}(x^{(1)},\nu)=0$,
and the complementarity conditions are reduced to 
\[x_j(1-x_j)F^{(1)}_{j}(x^{(1)},\nu)=0, \quad x_{2+j}(1-x_{2+j})F^{(2)}_{j}(x^{(2)},\nu)=0,\quad (j=1,2).\]
Consequently, the CS-LME reformulation to (\ref{eq:box}) is
\be\label{eq:box-cslme-reform}
\left\{\begin{array}{llll}
\displaystyle\min_{x\in \R^{4}} &x_1^4x_2^2+x_1^2x_2^4+x_3^6-3x_1^2x_2^2x_3^2+x_3^3+x_3x_4^2-2x_3^2x_4\\
\enspace \mathrm{s.t.}  &  x_j(1-x_j)F^{(1)}_{j}(x^{(1)},\nu)=0, \  x_{2+j}(1-x_{2+j})F^{(2)}_{j}(x^{(2)},\nu)=0, & (j=1,2)\\
& (1-x_j)F^{(1)}_{j}(x^{(1)},\nu)\ge0,\  (1-x_{2+j})F^{(2)}_{j}(x^{(2)},\nu)\ge0, & (j=1,2)\\
& -x_j F^{(1)}_{j}(x^{(1)},\nu)\ge0,\  -x_{2+j}F^{(2)}_{j}(x^{(2)},\nu)\ge0, & (j=1,2)\\
& F^{(1)}_{3}(x^{(1)},\nu)=0,\  0\le x_1\ddd x_4\le 1.
\end{array}\right.\ee
Later in Section~\ref{sc:ne}, we will compare the numerical performance of solving the CS-LME reformulation~\eqref{eq:box-cslme-reform} of \eqref{eq:box} with solving it directly and solving its LME reformulation~\eqref{eq:pblme}, all using {\tt CS-TSSOS} \cite{CSTSSOS}.
\end{example}}

To summarize,
the LME approach proposed by Nie~\cite{nie2019tight} allows for tightening the classical Moment-SOS relaxation by incorporating necessary polynomial constraints, provided that certain nonsingularity conditions hold.
However, usually this approach cannot keep the csp from the original polynomial optimization problem.
Moreover, when the nonsingularity condition fails, LMEs do not exist.
In contrast, one can try to find the CS-LME instead by adding some new variables.
In the above, we demonstrate how to find CS-LMEs for the two-block cases.
In the next subsection, we provide a systematic way to construct CS-LMEs for an arbitrary number of blocks.

\subsection{Correlatively sparse LME: multi-blocks}\label{subsec:cslmemb}
We introduce how to construct CS-LMEs for an arbitrary number of blocks in this subsection. Hereafter, we assume that the correlative sparsity pattern $(\cI_1,\ldots,\cI_s)$ satisfies the RIP. 
Without loss of generality, we also assume the following conditions hold: 
\begin{itemize}
\item [1.]$\cI_i $ is not included in $ \cI_j $  for any two distinct  $i,j\in [s]$;
\item [2.] $\cI_{i+1}\cap \cup_{j=1}^i \cI_j \neq \emptyset $ for any $i\in [s-1]$.
\end{itemize}
We remark that under the RIP condition, the second condition always holds unless there exists a proper subset $S$ of $[s]$ such that $\left(\cup_{i\in S} \cI_i\right)\cap\left(\cup_{i\notin S} \cI_i\right)=\emptyset$ for which we can solve the polynomial optimization problems for variables within $S$ and outside of $S$ separately.

To construct CS-LME coherent with the csp $(\cI_1,\ldots,\cI_s)$, 
we first build a directed tree with nodes corresponding to the elements in the collection  $\{\mc{I}_1,\dots, \mc{I}_s\}$. 
\begin{algorithm}[H]
\caption{Clique Tree Construction}
\label{A:TC}
\begin{algorithmic}[1]
\Require   $(\mc{I}_1,\dots,\mc{I}_s)$ satisfying the RIP.
\State $V=\{1,\ldots,s\}$ and $A=\emptyset$.
\For {$i=1, \ldots, s-1$}
\If  {$\mc{I}_{i+1}\bigcap \bigcup_{j=1}^{i} \mc{I}_j \neq \emptyset$}
\State Find the largest $t\leq i$ such that $\mc{I}_{i+1}\bigcap \bigcup_{j=1}^{i} \mc{I}_j
\subseteq \mc{I}_t$.
\State $A=A\bigcup \{\left( {i+1}, {t}\right )\}$.
\EndIf
\EndFor
\Ensure ${G(V,A)}$
\end{algorithmic}
\end{algorithm}
The \textit{csp graph} associated with~\eqref{eq:cs-polyopt} is the undirected graph $G^{csp}=G(W,E)$,
with nodes $W=[n]$ and edges $E$ satisfying $\{k_1,k_2\}\in E$ if there exists $i\in [s]$ such that $k_1\in \cI_i$ and $k_2\in \cI_i$.  Since $(\cI_1,\ldots,\cI_s)$ satisfies the RIP,  the corresponding csp graph $G^{csp}$ is chordal\footnote{ A graph is chordal if all its cycles of length at least four have an edge that joins two nonconsecutive nodes. } and $\{\cI_1,\ldots,\cI_s\}$ is  the list of maximal cliques of $G^{csp}$,
because we assumed that $\cI_i$ is not contained in $\cI_j$ for any distinct $i,j\in [s]$.
A \textit{clique tree} of the graph $G^{csp}$ is a tree on the set $V=[s]$ such that for every pair of distinct 
nodes $i,j\in [s]$,  we have $\cI_i \cap \cI_j \subseteq  \cI_k$ for any $k\in [s]$ on the path connecting $i$ and $j$ in the tree. 
Clique tree exists because $G^{csp}$ is chordal; see~\cite[Theorem 3.1]{BP93}. 
The output $G(V,A)$ of~\Cref{A:TC} is a directed tree whose underlying undirected graph is a clique tree of the graph $G^{csp}$. This follows from~\cite[Theorem 3.4]{BP93}.  The directions indicate the ``parent-child" relation between cliques on the tree. 
We refer to~\cite{BP93} for more details on chordal graphs and clique trees.

Given the clique tree $G(V,A)$ produced by \Cref{A:TC},
for each $i\in [s]$,
we denote the indices of children of the node $i$ by 
\be \label{eq:Di}
\cD_i:=\{t: \left({t}, {i}\right)\in A\},
\ee
and the index of the parent of node $i$ by
\be \label{eq:Ai}
\cA_i:=\{t: \left({i},{t}\right)\in A\}.
\ee
For each $i\in \{2,\ldots,s\}$, $\cD_i$ can be empty sets
and $\cA_i$ contains exactly
one element. 
When $\left(i, t\right)\in A$, we let
\be\label{eq:Cit}
\cC_{i,t}:=\mc{I}_i \bigcap \mc{I}_t
\ee
be the indices of all variables shared by block $i$ and block $t$. 
Then, we introduce a group of auxiliary variables:
\be\label{eq:nuitk}
\{\nu_{i,t,k}: (i,t)\in A, k\in \cC_{i,t} \}.
\ee
In other words, for each arc $(i,t)\in A$, we need the same number of auxiliary variables as the number of variables shared by the block $i$ and block $t$.
For every $i\in [s]$, 
define
\be \label{eq:cJ}
\cJ_i:=\left\{(i,t,k): t\in \cA_i,\, k \in \cC_{i,t}\right\} \cup \left\{(t,i,k): {t}\in \cD_i,\, k \in \cC_{t,i}\right\},
\ee
and (recall that the vector $e^{(i)}_k$ is defined in (\ref{eq:eik}))
\be\label{eq:nui}
\nu^{(i)}:=
- \sum_{t\in \cA_i} \sum_{k\in\mc{C}_{i,t}}  \nu_{i,t,k} \ei^{(i)}_k +\sum_{t\in \cD_i} \sum_{k\in\mc{C}_{t,i}}  \nu_{t,i,k} \ei^{(i)}_k \in \re^{n_i}.
\ee
Clearly, the vector $\nu^{(i)}$ only depends on variables {in the group~\eqref{eq:nuitk} }indexed by $\cJ_i$ for each $i\in [s]$. We illustrate how to construct new variables in the following example.
\begin{example}\label{ex:cscser}
Consider the following csp pattern:
\be\label{a:excspp1}
\begin{array}{c}
\cI_1=\{1,2,3,4\},\enspace \cI_2=\{1,2,5,6\},\enspace \cI_3=\{1,2,7,8\},
\\
\enspace\cI_4=\{1,2,9,10\}, \enspace \cI_5=\{1,2,11,12\}.
\end{array}\ee
Then the set of edges $A$ in the clique tree $G(V,A)$ produced by \Cref{A:TC} is 
\[A=\{(2,1),\ (3,2),\ (4,3),\ (5,4)\},\]
and
$\cD_i=\{i+1\}$ for each $i=1\ddd 4$, $\cA_i=\{i-1\}$ for each $ i=2\ddd 5$.
Thus 
\[\begin{aligned}
\cJ_1&=\{(2,1,1),\ (2,1,2)\},\\
\cJ_2&=\{(2,1,1),\ (2,1,2)\}\cup\{(3,2,1),\ (3,2,2)\},\\
\cJ_3&=\{(3,2,1),\ (3,2,2)\}\cup\{(4,3,1),\ (4,3,2)\},\\
\cJ_4&=\{(4,3,1),\ (4,3,2)\}\cup\{(5,4,1),\ (5,4,2)\},\\
\cJ_5&=\{(5,4,1),\ (5,4,2)\}.
\end{aligned}\]
An illustration of the directed tree obtained from~\Cref{A:TC} and auxiliary variables are given in~\Cref{fig:n5}. 
For this clique tree, we have $|\cJ_1|=|\cJ_5|=2$ and $|\cJ_2|=|\cJ_3|=|\cJ_4|=4$.
\end{example}
\begin{figure}[!ht]
\centering
\begin{tikzpicture}[node distance=3cm]
\node (n1) [nodes1] {$\cI_1$};
\node (n2) [nodes1,right of=n1]{$\cI_2$};
\node (n3) [nodes1,right of=n2]{$\cI_3$};
\node (n4) [nodes1,right of=n3]{$\cI_4$};
\node (n5) [nodes1,right of=n4]{$\cI_5$};
\draw [arrow] (n2) -- node[anchor=south]{$\nu_{2,1,1}$}(n1);
\draw [arrow] (n3) -- node[anchor=south]{$\nu_{3,2,1}$}(n2);
\draw [arrow] (n4) -- node[anchor=south]{$\nu_{4,3,1}$}(n3);
\draw [arrow] (n5) -- node[anchor=south]{$\nu_{5,4,1}$}(n4);
\draw [arrow] (n5) -- node[anchor=north]{$\nu_{5,4,2}$}(n4);
\draw [arrow] (n2) -- node[anchor=north]{$\nu_{2,1,2}$}(n1);
\draw [arrow] (n3) -- node[anchor=north]{$\nu_{3,2,2}$}(n2);
\draw [arrow] (n4) -- node[anchor=north]{$\nu_{4,3,2}$}(n3);
\end{tikzpicture}
\caption{Clique tree returned by~\Cref{A:TC} and auxiliary variables for the csp pattern~\eqref{a:excspp1}.}
\label{fig:n5}
\end{figure}
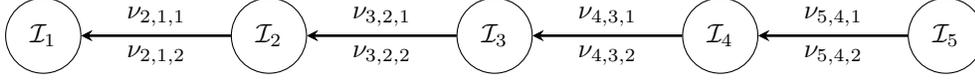

With new variables $\nu_{i,t,k}$ and vectors $\nu^{(i)}$ given by (\ref{eq:nui}),
we rewrite the KKT system~\eqref{eq:cs-KKT}.
For each $i\in [s]$, consider the following system on $(x^{(i)}, \lambda^{(i)},\nu^{(i)})\in \R^{n_i+m_i+\ell_i+|\cJ_i|}$:
\be\label{eq:sub-KKT}
\left\{\begin{array}{c}\nabla_i f_i (x^{(i)})+\nu^{(i)} = \displaystyle\sum_{j=1}^{m_i}\lmd^{(i)}_{j}\nabla_i g^{(i)}_j(x^{(i)})+\sum_{j=1}^{\ell_i}\lambda^{(i)}_{m_i+j} \nabla_i h^{(i)}_{j}(x^{(i)}),\\
h^{(i)}(x^{(i)})=0,\\[4pt]
0\le \lmd_{1:m_i}^{(i)}\perp g^{(i)}(x^{(i)})\ge 0.
\end{array}
\right.
\ee
\begin{proposition}\label{prop:1}
Let $x:=(x_1\ddd x_n)\in \R^n$ and $\lmd:=(\lmd^{(1)}\ddd \lmd^{(s)}) \in \R^{m+\ell}$. The pair $(x,\lmd)$ is a KKT pair of~\eqref{eq:cs-polyopt} if and only if there exists a group of auxiliary variables $\{\nu_{i,t,k}: (i,t)\in A, k\in \cC_{i,t} \}$ such that \eqref{eq:sub-KKT} holds for all $i\in [s]$.
\end{proposition}
\begin{proof}
By lifting all the vectors into $\R^n$ {(i.e., filling in $0$ to the coordinates that are not in $\cI_i$)},
we can rewrite  the first equation in~\eqref{eq:sub-KKT} as
\be\label{eq:lifteqkkt1}
\nabla f_i (x){+\hat\nu^{(i)} } = \displaystyle\sum_{j=1}^{m_i}\lmd^{(i)}_{j}\nabla g^{(i)}_j(x)+\sum_{j=1}^{\ell_i}\lambda^{(i)}_{m_i+j} \nabla h^{(i)}_{j}(x),\ee
where $\hat\nu^{(i)}\in \R^n$ is obtained by lifting $\nu ^{(i)}$ into $\R^n$:
\[\hat{\nu}^{(i)}:=-\sum_{t\in\cA_i}\sum_{k\in\mc{C}_{i,t}}  \nu_{i,t,k} \ei_k +\sum_{t\in \cD_i} \sum_{k\in\mc{C}_{t,i}}  \nu_{t,i,k} \ei_k.\]
If there exists $(x,\lambda)$ and $\{\nu_{i,t,k}: (i,t)\in A, k\in \cC_{i,t} \}$ such that~\eqref{eq:sub-KKT} holds for all $i\in [s]$,
then 
\begin{align*}
\nabla f(x) & \overset{\qquad}{=} \sum_{i=1}^s  \nabla f_i (x)\\
&\overset{\eqref{eq:lifteqkkt1}}{=} \sum_{i=1}^s \left( \sum_{j=1}^{m_i}\lmd^{(i)}_{j}\nabla g^{(i)}_j(x)+\sum_{j=1}^{\ell_i}\lambda^{(i)}_{m_i+j} \nabla h^{(i)}_{j}(x) - \hat \nu^{(i)}\right)
\\& \overset{\qquad}{=} \sum_{i=1}^s \left(\sum_{j=1}^{m_i}\lmd^{(i)}_{j} \nabla g^{(i)}_{j}(x)+ \sum_{j=1}^{\ell_i}\lambda^{(i)}_{m_i+j} \nabla h^{(i)}_{j}(x)
\right)-\sum_{i=1}^s \hat{\nu}^{(i)}\\
&\overset{\qquad}{=} \sum_{i=1}^s \left(\sum_{j=1}^{m_i}\lmd^{(i)}_{j} \nabla g^{(i)}_{j}(x)+ \sum_{j=1}^{\ell_i}\lambda^{(i)}_{m_i+j} \nabla h^{(i)}_{j}(x)
\right).
\end{align*}
Therefore, $(x,\lmd)$ is a KKT pair of~\eqref{eq:cs-polyopt}.
In the following, we show the other direction.

Let $(x,\lmd)$ be a KKT pair of~\eqref{eq:cs-polyopt}.
{For each fixed $k\in [n]$, denote 
\[\cP_k:=\{(i, t)\in A: k\in \cC_{i,t}\},\quad \cQ_k:=\{i: k\in \mc{I}_i\}.\]
In other words, $\cQ_k$ corresponds to the set of cliques that contain $k$ and $G(\cQ_k, \cP_k)$ is the subgraph of $G(V,A)$ induced by the nodes $\cQ_k$. 
Then by \cite[Theorem 3.2]{BP93}, for each $k\in [n]$, the underlying undirected graph of $G(\cQ_k, \cP_k)$  is a tree.}

This allows us to deduce the solvability of the following system of linear equations for each fixed $k\in [n]$:
{\be\label{a:psdd}\begin{array}{l}
\displaystyle \sum_{\substack{t\in\cP_k(i)} } \nu_{i,t,k}- \sum_{\substack{t\in \cP'_k(i)}} \nu_{t,i,k}\\ 
\displaystyle \qquad\qquad
=\frac{\partial f_i}{\partial x_k}(x)-\left(
\sum_{j=1}^{m_i}\lmd^{(i)}_{j} \frac{\partial g^{(i)}_j}{\partial x_k}(x)+ \sum_{j=1}^{\ell_i}\lambda^{(i)}_{m_i+j} \frac{\partial h^{(i)}_j}{\partial x_k}(x)\right),\enspace \forall i\in \cQ_k.
\end{array}\ee
In the above,
\be\label{eq:cpki} \cP_k(i)\deq\{t : \left (i,t\right )\in \cP_k\}, \quad 
\cP'_k(i)\deq\{t : \left (t,i\right )\in \cP_k\}.\ee}
Indeed, the linear system~\eqref{a:psdd} can be written as 
\begin{align}\label{a:Bvb}
Bv=b,
\end{align}
where $B\in \R^{|\cQ_k|\times |\cP_k|}$ is the incidence matrix of $(\cQ_k, \cP_k)$ and $b\in \R^{|\cQ_k|}$ is a vector satisfying ${\bf{1}}^\top b=0$. 
Since {the underlying undirected graph of $G(\cQ_k, \cP_k)$ is a tree}, we have $\rank(B)=|\cP_k|=|\cQ_k|-1$ and
${\bf{1}}^\top  B=0$. Therefore~\eqref{a:Bvb}, and thus~\eqref{a:psdd} for each $k\in[n]$, are solvable.
In other words, there exist $\{\nu_{i,t,k}: (i,t)\in A,\, k\in \cC_{i,t} \}$ such that \eqref{a:psdd} holds for all $k\in [n]$. So  the following equations hold at $(x,\lmd)$:
\be\label{a:sfwea}
{
\begin{array}{ll}
&\displaystyle  \sum_{k\in\cI_i} 
\sum_{\substack{t\in\cP_k(i)} } \nu_{i,t,k} \ei_k
-\sum_{k\in\cI_i} \sum_{\substack{ t\in\cP'_k(i)} } \nu_{t,i,k} \ei_k \\
 = \quad &\displaystyle \sum_{k\in \cI_i} \left(\frac{\partial f_i}{\partial x_k}(x) - \sum_{j=1}^{m_i}\lmd^{(i)}_{j} \frac{\partial g^{(i)}_j}{\partial x_k}(x)- \sum_{j=1}^{\ell_i}\lambda^{(i)}_{m_i+j} \frac{\partial h^{(i)}_j}{\partial x_k}(x)\right) \ei_k,\enspace \forall i\in [s]. 
\end{array}
}
\ee
Note that for any $k\notin \cI_i$, we have 
$$
\frac{\partial f_i}{\partial x_k}(x)\equiv \frac{\partial g^{(i)}_1}{\partial x_k}(x)\equiv\dots \equiv  \frac{\partial g^{(i)}_{m_i}}{\partial x_k}(x)\equiv\frac{\partial h^{(i)}_1}{\partial x_k}(x)\equiv \dots \equiv \frac{\partial h^{(i)}_{\ell_i}}{\partial x_k}(x)\equiv0.
$$
{Therefore,~\eqref{a:sfwea} yields that for each  $i\in[s]$,
\be\label{a:sfwea2}
\begin{aligned}&
  \sum_{k\in \cI_i}\left(\sum_{\substack{t\in\cP_k(i)} } \nu_{i,t,k}
- \sum_{\substack{t\in\cP^{\prime}_k(i)}} \nu_{t,i,k}\right) \ei_k\\
=\quad &\sum_{k\in [n]} \left(\frac{\partial f_i}{\partial x_k}(x)-\sum_{j=1}^{m_i}\lmd^{(i)}_{j} \frac{\partial g^{(i)}_j}{\partial x_k}(x)- \sum_{j=1}^{\ell_i}\lambda^{(i)}_{m_i+j} \frac{\partial h^{(i)}_j}{\partial x_k}(x)\right) \ei_k  \\
= \quad & \nabla f_i(x)- 
\sum_{j=1}^{m_i}\lmd^{(i)}_{j} \nabla g^{(i)}_{j}(x) - \sum_{j=1}^{\ell_i}\lambda^{(i)}_{m_i+j} \nabla h^{(i)}_{j}(x). 
\end{aligned}
\ee
In light of \eqref{eq:Di}--\eqref{eq:Ai}, for each fixed $i\in [s]$, we have
$$
\{(t,k): k\in \cI_i, t\in\cP_k(i)\}=\{(t,k): t\in \cA_i, k\in \cC_{i,t}\}.
$$
 We then obtain 
\begin{align*}
&\sum_{k\in\cI_i}\left(\sum_{\substack{t\in\cP_k(i)} } \nu_{i,t,k}
- \sum_{\substack{t\in\cP^{\prime}_k(i)}} \nu_{t,i,k}\right) \ei_k \\
=\quad & \sum_{t\in\cA_i}\sum_{k\in\mc{C}_{i,t}}  \nu_{i,t,k} \ei_k -\sum_{t\in \cD_i} \sum_{k\in\mc{C}_{t,i}}  \nu_{t,i,k} \ei_k
\\
= \quad &-\hat\nu^{(i)}.
\end{align*}
Therefore,~\eqref{eq:lifteqkkt1} holds, and the first equation in~\eqref{eq:sub-KKT} is satisfied. }
\end{proof}
{
\begin{remark}\label{r:resaAD}
In \Cref{A:TC}, even if we replace line 4  by
\be\label{line4algo}
\mathrm{Find~an~arbitrary~} t\leq i \mathrm{~such~that~}\mc{I}_{i+1}\bigcap \bigcup_{j=1}^{i} \mc{I}_j
\subseteq \mc{I}_t,
\ee 
the resulting tree is still a clique tree,
and the induced subtree property still holds for $G(V,A)$.
Hence,~\Cref{prop:1}, as well as all the results that will follow, still hold if line 4 of~\Cref{A:TC} is replaced with~\eqref{line4algo}.
This is because the only key property of $G(V,A)$ needed in the proof of~\Cref{prop:1} is the {\it induced subtree property} (see \cite[Theorem 3.2]{BP93}) satisfied by the clique tree. 
However, using an arbitrary $t$ as in \eqref{line4algo} may create a large number of children for some nodes (see~\Cref{ep:why-use-largest-t} below),
which will increase the number of variables in $\nu^{(i)}$ (hence the size of blocks and the computational cost). 
In other words, one would prefer a tree with a large depth and small breadth.
That is why we propose choosing the largest $t$ in~\Cref{A:TC}.
\end{remark}
\begin{example}\label{ep:why-use-largest-t}
Consider the csp pattern~\eqref{a:excspp1} again. 
If we use~\eqref{line4algo} to replace line 4 in \Cref{A:TC},
then another possible directed clique tree and auxiliary variables is shown in~\Cref{fig:n53}.
For this clique tree, we have $|\cJ_1|=8$, $|\cJ_2|=|\cJ_3|=|\cJ_4|=|\cJ_5|=2$.
\end{example}
}
\begin{figure}[!ht]
\centering
\begin{tikzpicture}[node distance=3cm]
\node (n1) [nodes1] {$\cI_2$};
\node (n2) [nodes1,right of=n1]{$\cI_1$};
\node (n3) [nodes1,right of=n2]{$\cI_3$};
\node (n4) [nodes1,below of=n1,node distance=2cm]{$\cI_4$};
\node (n5) [nodes1,below of=n3,node distance=2cm]{$\cI_5$};
\draw [arrow] (n1) -- node[anchor=south]{$\nu_{2,1,1}$}(n2);
\draw [arrow] (n3) -- node[anchor=south]{$\nu_{3,1,1}$}(n2);
\draw [arrow] (n4) -- node[anchor=east]{$\nu_{4,1,1}$}(n2);
\draw [arrow] (n5) -- node[anchor=south]{$\nu_{5,1,1}$}(n2);
\draw [arrow] (n5) -- node[anchor=north]{$\nu_{5,1,2}$}(n2);
\draw [arrow] (n1) -- node[anchor=north]{$\nu_{2,1,2}$}(n2);
\draw [arrow] (n3) -- node[anchor=north]{$\nu_{3,1,2}$}(n2);
\draw [arrow] (n4) -- node[anchor=west]{$\nu_{4,1,2}$}(n2);
\end{tikzpicture}
\caption{Another possible clique tree and auxiliary variables for the csp pattern~\eqref{a:excspp1}.}
\label{fig:n53}
\end{figure}
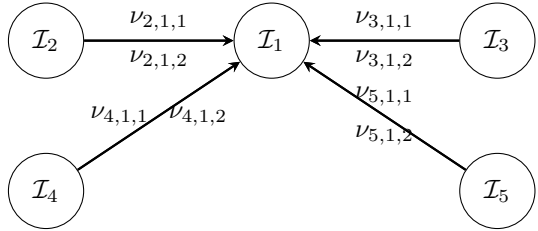

Under \Cref{ass:lme},~\eqref{eq:sub-KKT} implies that the $i$th group of Lagrange multipliers can be expressed by a tuple of polynomials which only depends on variables indexed by $\cI_i$ and $\cJ_i$, say, $x^{(i)}$ and $\nu^{(i)}$ {(by abuse of notation, here $\nu^{(i)}$ means the tuple of all variables involved in the vector $\nu^{(i)}$).
We let 
\be\label{eq:Fi}
z^{(i)}\deq (x^{(i)},\nu^{(i)}),\quad F^{(i)}(z^{(i)}) \deq \nabla_if_i(x^{(i)})+\nu^{(i)}.
\ee}

\begin{theorem}\label{thm:main}
Under~\Cref{ass:lme}, a vector $x\in \R^n$ is a KKT point of~\eqref{eq:cs-polyopt} if and only if the following system~\eqref{eq:sub-KKTLME} holds for each $i\in [s]$:
\be\label{eq:sub-KKTLME}
\left\{\begin{array}{ccc}
F^{(i)}(z^{(i)})= \sum_{j=1}^{m_i} p^{(i)}_{j}\left(z^{(i)}\right)\nabla_i g^{(i)}_j(x^{(i)})+\sum_{j=1}^{\ell_i}p^{(i)}_{m_i+j}\left(z^{(i)}\right)\nabla_i h^{(i)}_{j}(x^{(i)}),\\
0\le p_{1:m_i}^{(i)}(z^{(i)})\perp g^{(i)}(x^{(i)})\ge 0,\ h^{(i)}(x^{(i)})\ge 0,
\end{array}
\right.
\ee
where
\be\label{a:pilme}
p^{(i)}(z^{(i)}) \deq L^{(i)}(x^{(i)})\cdot F^{(i)}(z^{(i)}),
\ee
and $z^{(i)}$ and $F^{(i)}$ are defined in~\eqref{eq:Fi}.
\end{theorem}
\begin{proof}
Recall the matrix of polynomials $\bG^{(i)}(x^{(i)})$ defined in~\eqref{eq:gisedf}.
The system~\eqref{eq:sub-KKT} is equivalent to
\begin{align}\label{a:Ggilambda}
\left\{ 
\begin{array}{c}
\bG^{(i)}(x^{(i)}) \lmd^{(i)}=\left [\,
F^{(i)}(z^{(i)})^{\top}\  0\  \cdots\ 0
\, \right ]^{\top},  \\
\lambda^{(i)}_{1},\ldots, \lambda^{(i)}_{m_i}\geq 0,\ 
g^{(i)}_{1},\ldots, g^{(i)}_{m_i}\geq 0, \ 
h^{(i)}_{1},\ldots, h^{(i)}_{\ell_i}\geq 0, 
\end{array}
\right.
\end{align}
By~\Cref{ass:lme}, the first equation in \eqref{a:Ggilambda} holds if and only if
\[\lambda^{(i)} = L^{(i)}(x^{(i)})\cdot F^{(i)}(z^{(i)}).\]
Thus, it remains to replace $\lambda^{(i)}$ with $p^{(i)}(z^{(i)})$ in~\eqref{eq:sub-KKT} and apply~\Cref{prop:1}.
\end{proof}

\begin{remark}
For the polynomial optimization problem~\eqref{eq:cs-polyopt} with general csp $(\cI_1\ddd \cI_s)$,
we call the vector of polynomials $p^{(i)}(z^{(i)})$ defined in \eqref{a:pilme} the CS-LMEs of $\lmd^{(i)}$.
By Proposition~\ref{prop:1} and Theorem~\ref{thm:main},
CS-LMEs exist when Assumption~\ref{ass:lme} is satisfied.
\end{remark}

\subsection{A CS-LME reformulation}

{
In the rest of this paper, we give a CS-LME reformulation for the polynomial optimization problem~\eqref{eq:cs-polyopt} under~\Cref{ass:lme}.
For each $i\in [s]$, denote 
$$\phi^{(i)}(z^{(i)}):= 
\left[ \begin{array}{c} F^{(i)}(z^{(i)})-\displaystyle\sum_{j=1}^{m_i}p^{(i)}_{j}(z^{(i)})\nabla_i g^{(i)}_j(x^{(i)})-\sum_{j=1}^{\ell_i} p^{(i)}_{m_i+j}(z^{(i)})\nabla_i h^{(i)}_{j}(x^{(i)}) \\
h^{(i)}(x^{(i)})\\
p_{1:m}^{(i)}(z^{(i)}) \circ g^{(i)}(x^{(i)})
\end{array}
\right ],
$$ 
and
$$ 
\psi^{(i)}(z^{(i)}):= \left[\begin{array}{c} p_{1:m}^{(i)}(z^{(i)})  \\
g^{(i)}(x^{(i)})
\end{array}
\right ].
$$
Here, the polynomial $p_j^{(i)}\in \R[z^{(i)}]$ is the $j$th entry of the CS-LME $p^{(i)}$ defined in~\eqref{a:pilme}.
Based on~\Cref{thm:main}, we 
propose the following CS-LME typed reformulation of~\eqref{eq:cs-polyopt}:
\begin{equation}\label{eq:pblmesparse}
\left\{\begin{array}{lcl}
\displaystyle f_c:=&\displaystyle \min_{z^{(1)}\ddd z^{(s)}} ~~~ &f_1(x^{(1)})+\dots+f_s(x^{(s)})\\
&\enspace \mathrm{s.t.} \quad  &
\psi^{(i)}(z^{(i)})\geq 0,\enspace  i \in [s]\\ &&
\phi^{(i)}(z^{(i)})=0,\enspace i\in [s] 
\end{array}\right.
\end{equation}
The previous reformulation~\eqref{eq:pblme-2} for the case $s=2$ is a special case of (\ref{eq:pblmesparse}).}
If we let
\be\label{eq:hatI}
\hat{\cI}_i:=\mc{I}_i\bigcup  \cJ_i,\ee
then~\eqref{eq:pblmesparse} has the csp $(\hat{\cI}_1\ddd \hat{\cI}_s)$. 
Suppose the global minimum $f_{\min}$ of (\ref{eq:cs-polyopt}) is attained at some KKT point,
then at least one minimizer of (\ref{eq:cs-polyopt}) is feasible for (\ref{eq:pblmesparse}), thus $f_{\min}\ge f_c$.
Since the feasible set of (\ref{eq:pblmesparse}) is contained in the feasible set of (\ref{eq:cs-polyopt}), we have $f_{\min}\le f_c$.
So, we conclude the following from the statement above.
\begin{theorem}\label{thm:sec3}
If the minimum $f_{\min}$ of (\ref{eq:cs-polyopt}) is attained at a KKT point,
then the minimal value (\ref{eq:cs-polyopt}) and (\ref{eq:pblmesparse}) are identical, i.e., $f_{\min}= f_c$.
\end{theorem}
\begin{remark}
Suppose the minimum value $f_{\min}$ is attainable.
If the nonsingularity condition holds for (\ref{eq:cs-polyopt}),
then $f_{\min}$ is attained at KKT points,
since the nonsinglarity implies the linear independence constraint qualification conditions (LICQ) hold on $\mathbb{C}^n$.
However, this is not necessarily true if we replace the nonsingularity condition of $(g,h)$ by that of every $(g^{(i)},h^{(i)})$, i.e., Assumption~\ref{ass:lme}, since \Cref{ass:lme} does not guarantee the LICQ to hold at every feasible point.
For such cases, $f_c$ may or may not equal $f_{\min}$.
Nevertheless, it does not mean the KKT conditions must fail at minimizers of (\ref{eq:cs-polyopt}) if the nonsingularity condition does not hold.
Indeed, it may happen that the constraining tuple is singular, but the LICQ condition holds at a minimizer, thus $f_c=f_{\min}$; see~\Cref{ep:box-ne}(ii), \Cref{ep:nonaive} and~\Cref{ep:ls}. 
\end{remark}


\section{Correlatively sparse LME based SOS relaxation}
\label{sc:CSSOS}

This section studies the correlatively sparse SOS relaxations for solving the CS-LME reformulation~\eqref{eq:pblmesparse}.

\subsection{RIP of the CS-LME reformulation}
First, we establish the RIP for \eqref{eq:pblmesparse}.
Recall that for each $i\in[s]$, the set of indices of variables $\hat{\cI}_i$ is given in (\ref{eq:hatI}).
\begin{lemma}\label{lemma:hatI}
The csp $(\hat{\cI}_1,\dots, \hat{\cI}_s)$ satisfies the RIP in \Cref{def:rip}.
\end{lemma}
\begin{proof}
Note that
$$
\begin{aligned}
  & \cJ_{i+1}\bigcap \bigcup_{j=1}^i \cJ_j \\
=\  &  \left( \left\{(i+1,t,k): t\in \cA_{i+1}, k\in \cC_{i+1,t}\right\} \, \bigcup \, \left\{(t,i+1,k): t\in \cD_{i+1}, k\in \cC_{t,i+1}\right\}\right)\\ 
& \qquad\qquad \bigcap  \bigcup_{j=1}^i  
\left(\{(j,t,k): t\in \cA_j, k\in \cC_{j,t}\} \,\bigcup\, \{(t,j,k): t\in \cD_j,k\in \cC_{t,j}\}\right).
\end{aligned}
$$
Since $t\in\cD_{i+1}$ implies $t>i+1$, and $t\in \cA_j$ implies $t<j$, we have
$$
\begin{aligned}
\cJ_{i+1}\bigcap \bigcup_{j=1}^i \cJ_j=\ & \{(i+1,t,k): t\in \cA_{i+1}, k\in \cC_{i+1,t}\}\bigcap \\
   & \bigcup_{j=1}^i \left( \{(j,t,k): t\in \cA_j, k\in \cC_{j,t}\}\bigcup \{(t,j,k): t\in \cD_j,\, k\in \cC_{t,j}\}\right)\\
\subseteq\ & \{(i+1,t,k): t\in \cA_{i+1}, k\in \cC_{i+1,t}\} . 
\end{aligned}
$$
Let $ \cA_{i+1}=\{t\}$ for some $t\in [s]$. Then  $i+1\in \cD_t$ and so
\[\begin{aligned}
\cJ_t&=   \left\{(t,i,k): i\in \cA_{t}, k\in \cC_{t,i}\right\}\cup  
\left\{(i,t,k): i\in \cD_{t}, k\in \cC_{i,t}\right\}  \\ &
\supseteq \left\{(i+1,t,k): k\in \cC_{i+1,t}\right\}.
\end{aligned}
\]
Note that $\cI_i$ is the set of indices of variables $x^{(i)}$ and $\cJ_i$ is the set of indices of the auxiliary variables $\nu^{(i)}$. Hence $\cI_i \cap \cJ_j =\emptyset$ for each pair of $i,j\in [s]$. In particular, 
\begin{align*}
&\hat{\cI}_{i+1}\bigcap \bigcup_{j=1}^{i} \hat{\cI}_j  \\
= \quad &\left\{\mc{I}_{i+1} \bigcup \cJ_{i+1} \right\}\bigcap \left\{\bigcup_{j=1}^{i}  \left(\mc{I}_j \bigcup \cJ_j \right)\right\}\\
=\quad &\left\{\mc{I}_{i+1}\bigcap  \left\{\bigcup_{j=1}^{i}  \left(\mc{I}_j \bigcup \cJ_j \right)\right\} \right\}   \bigcup  \left\{\cJ_{i+1} \bigcap \left\{\bigcup_{j=1}^{i}  \left(\mc{I}_j \bigcup \cJ_j \right)\right\} \right\}\\
=\quad &\left\{\mc{I}_{i+1}\bigcap \bigcup_{j=1}^{i} \mc{I}_j\right\} \bigcup \left\{ \cJ_{i+1}\bigcap \bigcup_{j=1}^{i} \cJ_j\right\}.
\end{align*}
Therefore, we have
\begin{align*}
\hat{\cI}_{i+1}\bigcap \bigcup_{j=1}^{i} \hat{\cI}_j  
=\left\{\mc{I}_{i+1}\bigcap \bigcup_{j=1}^{i} \mc{I}_j\right\} \bigcup \left\{ \cJ_{i+1}\bigcap \bigcup_{j=1}^{i} \cJ_j\right\}
\subseteq \mc{I}_t \cup \cJ_t= \hat{\cI}_t.
\end{align*}

\end{proof}

\subsection{Convergence of the CS-LME based SOS relaxation}
For the polynomial optimization problem \eqref{eq:pblmesparse} with the csp~$({\hat\cI_1}, \ldots, \hat{\cI}_s)$,
the $d$th order correlatively sparse SOS relaxation is
\begin{equation}\label{eq:msoscslme}
\qquad
\left\{\begin{array}{lll}
\vartheta_{d}:=&\displaystyle\max &\gamma\\
&\enspace \mathrm{s.t.}   & {\displaystyle\sum_{i=1}^s f_i-\gamma \in  \displaystyle \sum_{i=1}^s\IQ_{\hat\cI_i} \left( \phi^{(i)}, \psi^{(i)}\right)_{2d}}.
\end{array}\right.
\end{equation}
Note that for each $i\in [s]$,  $h^{(i)}$ is contained in $\phi^{(i)}$, $g^{(i)}$ is contained in $\psi^{(i)}$,
and  $\cI_i \subseteq \hat{\cI}_i$. 
It follows that
$$
{\sum_{i=1}^s\IQ_{\hat\cI_i} \left( h^{(i)}, g^{(i)}\right)_{2d} \subseteq
\sum_{i=1}^s\IQ_{\hat\cI_i} \left( \phi^{(i)}, \psi^{(i)}\right)_{2d}}.
$$
Therefore,~\eqref{eq:msoscslme} is a tighter relaxation than~\eqref{eq:msoscs}. In particular, we have
\begin{align}\label{a:rhold}
\vartheta_d \geq \rho_d,\enspace \forall d\ge d_0.
\end{align}
\begin{theorem}
\label{tm:asymconv}
Assume the following
\begin{enumerate}
\item  at least one minimizer of~\eqref{eq:cs-polyopt} is a KKT point, and
\item for each $i\in [s]$, $\IQ_{{\cI_i}}\left(h^{(i)}, g^{(i)}\right)$ is archimedean.
\end{enumerate}
Then  
\begin{equation}\label{eq:rhold}
\lim_{d\rightarrow +\infty}\vartheta_d =f_{\min}.
\end{equation}
\end{theorem}
\begin{proof}
By the definition of CS-SOS relaxation, we have
$$
\vartheta_d\leq f_c,\enspace \forall d\in \bN.
$$
The first condition, together with~\Cref{thm:main}, implies that $f_c=f_{\min}$. 
Then we have $\vartheta_d\leq f_{\min}$,
and the convergence follows directly by (\ref{a:rhodlmt}) and~\eqref{a:rhold}.
\end{proof}
{
\begin{remark}\label{r:arche}
In Theorem~\ref{tm:asymconv},
if we substitute the  condition that $\IQ_{{\cI_i}}\left(h^{(i)}, g^{(i)}\right)$ is archimedean by the archimedeanness of $\IQ_{{\hat{\cI}_i}}\left(\phi^{(i)}, \psi^{(i)}\right)$,
then the conclusion still holds.
However,  $\IQ_{{\hat{\cI}_i}}\left(\phi^{(i)}, \psi^{(i)}\right)$ is not archimedean in general,
even if $\IQ_{{{\cI}_i}}\left(h^{(i)}, g^{(i)}\right)$ is archimedean. 
To see this, consider the CS-LME reformulation (\ref{eq:ball-cslme-reform}) for the optimization problem in \Cref{ep:box}. 
In (\ref{eq:ball-cslme-reform}),
tuples $h^{(1)},g^{(1)},h^{(2)},g^{(2)}$ are given by (\ref{eq:ball-csp}),
and it is clear that both $\IQ_{{{\cI}_1}}\left(h^{(1)}, g^{(1)}\right)$ and $\IQ_{{{\cI}_2}}\left(h^{(2)}, g^{(2)}\right)$ are archimedean.
Moreover,  $(\phi^{(1)}, \psi^{(1)})$ corresponds to the first two constraints in~\eqref{eq:ball-cslme-reform},
and $(\phi^{(2)}, \psi^{(2)})$ is given by the last two constraints in~\eqref{eq:ball-cslme-reform}. 
For any fixed $\nu\in \R$, consider the following polynomial optimization problem in variables $(x_1,x_2)$:
\begin{equation}\label{eq:ex1109}
\left\{\begin{array}{lll}
&\displaystyle\min ~~~ &f_1(x_1,x_2)+ \nu x_2\\
&\enspace \mathrm{s.t.} \quad  &
1-x_1^2 -x_2^2\geq 0.
\end{array}\right.
\end{equation}
Then one may check that $(x_1,x_2,\nu)\in \left\{z^{(1)}\in\re^3: \phi^{(1)}(z^{(1)})=0, \psi^{(1)}(z^{(1)})\ge0\right\}$ if and only if $(x_1,x_2)$ is a KKT point for \eqref{eq:ex1109}.
Since~\eqref{eq:ex1109} has a compact feasible set, and the constraint qualification condition holds at all feasible points, ~\eqref{eq:ex1109} has a KKT point for any $\nu\in\re$.
This implies that the semialgebraic set $$\left\{z^{(1)}\in\re^3: \phi^{(1)}(z^{(1)})=0, \psi^{(1)}(z^{(1)})\ge0\right\}$$ is unbounded,
and thus $\IQ_{{\hat{\cI}_1}}\left(\phi^{(1)}, \psi^{(1)}\right)$ is not archimedean.
Similarly, one can also show that $\IQ_{{\hat{\cI}_2}}\left(\phi^{(2)}, \psi^{(2)}\right)$ is not archimedean neither.
\end{remark}
\begin{remark}\label{r:posearchi}
The archimedean condition of $\IQ_{{\cI_i}}\left(h^{(i)}, g^{(i)}\right)$ for each $i\in [s]$ is also required in~\Cref{th:sparseputinar} to ensure the convergence of the CS-SOS relaxation. 
We wish to point out that this archimedean condition is not required for obtaining the CS-LMEs~\eqref{a:pilme} and the CS-LME reformulation~\eqref{eq:pblmesparse}.
There may exist polynomial optimization problems with compact feasible sets, for which,
however, $\IQ_{{{\cI}_i}}\left(h^{(i)}, g^{(i)}\right)$ is not archimedean for some $i\in[s]$ (e.g.,~\Cref{ep:box} and~\Cref{ex:box10}).
For such cases, one may add redundant constraints to $g^{(i)}$ to obtain the archimedeanness.
Such a redundant constraint can either be a replication of existing constraints,
or be the ball constraint as $ M-\Vert x^{(i)}\Vert^2\ge0$ if an {\it a priori}  bound $M$ is known.\footnote{It is important to note that there are two ways to replicate existing constraints.
For the constraint $g_j\in\re[x^{(i)}]$ that is not assigned to $g^{(i)}$,
we may add its replication to $g^{(i)}$ and obtain a new constraining tuple $\hat{g}^{(i)}$, then consider the KKT system and construct CS-LMEs for $\hat{g}^{(i)}$, as long as the new constraining tuple $\hat{g}^{(i)}$ is also nonsingular.
On the other hand, one may add $g_j$ to $\psi^{(i)}$ in the CS-LME reformulation.
These two ways produce different CS-LME reformulations with identical optimal values,
since the former may get different CS-LMEs from the original problem.
However, if we add a redundant ball constraint $ M-\Vert x^{(i)}\Vert^2\ge0$ which can never be active (e.g., let $M:=n_i\cdot\hat{M}$ with $\hat{M}>\Vert x^{(i)}\Vert_{\infty}$, thus its Lagrange multiplier must be $0$), then these two ways are equivalent.}
However, adding redundant constraints is inconvenient and usually unnecessary in practice.
Indeed, 
even if the archimedean conditions fail to hold (or, further, if the feasible set of \eqref{eq:cs-polyopt} is unbounded),
we can still formulate and solve the CS-LME reformulations with CS-SOS relaxations. 
In practical computation, finite convergence is observed numerically with a low relaxation order for solving CS-LME reformulations,
regardless of whether the archimedean condition for $\IQ_{{\cI_i}}\left(h^{(i)}, g^{(i)}\right)$ holds or not.
We refer to~\Cref{sc:ne} for examples where the archimedean condition is not satisfied,
while our approach can still find global minimum successfully. 
\end{remark}

\subsection{Comparison of the SDP problem scale}\label{subsec:cwlme}
In this section, we compare the scale of the corresponding SDP problems in different relaxation approaches. 
We assume that  the functions $f_1\in \R[x^{(1)}],\ldots,f_s\in \R[x^{s}]$ are all dense polynomials and both LMEs and CS-LMEs exist for~\eqref{eq:cs-polyopt}.  For the convenience of reference,
we nominate the four approaches for solving~\eqref{eq:cs-polyopt} as follows:
\begin{flushleft}
\begin{tabular}{p{0.135\textwidth} p{0.8\textwidth}}
 (SOS): &  Applying the dense SOS relaxation to~\eqref{eq:cs-polyopt}; \\
 (CS-SOS): &  Applying the CS-SOS relaxation to~\eqref{eq:cs-polyopt};\\
(LME): &  Applying the CS-SOS relaxation to the LME reformulation~\eqref{eq:pblme}; \\
(CS-LME): &  Applying the CS-SOS relaxation to the CS-LME reformulation~\eqref{eq:pblmesparse}. 
\end{tabular}
\end{flushleft}

We first consider the two-block cases.
Denote by $k\deq |\cC_{1,2}|$ the number of overlapping elements in $\cI_1$ and $\cI_2$. Then $|\cI_1 \cup \cI_2|=n_1+n_2-k$ is the total number of variables.
The CS-LME reformulation (\ref{eq:pblme-2}) has the csp $(\hat{\cI}_1,\hat{\cI}_2)$ such that $|\hat{\cI}_1|=n_1+k $ and $ |\hat{\cI}_2|=n_2+k.$ In~\Cref{tab:cmax}, we compare the maximal size of the positive semidefinite (PSD) matrices appearing in the SDP formulation of the four relaxation methods. 
\begin{table}[htb]
\centering
\caption{The maximal PSD matrix size in the $d$th order relaxation of the four  methods  when $s=2$. }
\begin{tabular}{c|l}  \hline
Relaxation approach  &   Maximal PSD matrix size in $d$th order relaxation \\
 \hline 
 SOS  &  $\textstyle\binom{n_1+n_2-k+d}{d} \times \binom{n_1+n_2-k+d}{d}$ \\
 \hline 
 CS-SOS &  $\textstyle\binom{\max\{n_1,n_2\}+d}{d}\times  \binom{\max\{n_1,n_2\}+d}{d}$ \\
 \hline 
 LME &  $\textstyle\binom{n_1+n_2-k+d}{d} \times  \binom{n_1+n_2-k+d}{d} $ \\
 \hline 
 CS-LME &  $\textstyle\binom{\max\{n_1,n_2\}+k+d}{d} \times \binom{\max\{n_1,n_2\}+k+d}{d} $ \\
\end{tabular}\label{tab:cmax}
\end{table}
In~\Cref{tab:B1B2}, we display the values of the binomial numbers in~\Cref{tab:cmax} for some examples of $n_1,n_2,k,d$.
\begin{table}[htb]
\centering
\caption{For each $n_1,n_2, k$, and $d$, we display  sequentially the  four binomial values appearing in~\Cref{tab:cmax}: $\binom{n_1+n_2-k+d}{d}$ for SOS, $\binom{\max\{n_1,n_2\}+d}{d}$ for CS-SOS, $\binom{n_1+n_2-k+d}{d}$ for LME, and $\binom{\max\{n_1,n_2\}+k+d}{d}$ for CS-LME.}
\begin{tabular}{c|c|c|c}
\hline 
$(n_1,n_2,k)$& $d=2$ & $d=3$ & $d=4$ \\
\hline
\scriptsize
$(4,3,1)$ &  \scriptsize $(28, 15, 28, 21)$  & \scriptsize (84, 35, 84, 56) & \scriptsize (210, 70, 210, 126) \\
\hline
\scriptsize $(5,5,2)$ & \scriptsize (45, 21, 45, 36) & \scriptsize (165, 56, 165, 120) & \scriptsize (495, 126, 495, 330)  \\
\hline
\scriptsize $(10,10,2)$ & \scriptsize  (190, 66, 190, 91)& \scriptsize (1330, 286, 1330, 455) & \scriptsize (7315, 1001, 7315, 1820) \\
\hline
\scriptsize $(15,15,3)$ & \scriptsize(406, 136, 406, 190) & \scriptsize(4060, 816, 4060, 1330) &\scriptsize (31465, 3876, 31465, 7315) \\
\hline
\scriptsize $(20,20,5)$ & \scriptsize (666, 231, 666, 351) & \scriptsize (8436, 1771, 8436, 3276) & \scriptsize (82251, 10626, 82251, 23751)
\end{tabular}
\label{tab:B1B2}
\end{table}

From~\Cref{tab:cmax} and~\Cref{tab:B1B2}, we conclude that for the same order of relaxation, the smallest scale SDP problem is given by CS-SOS. On the other hand, CS-SOS may need higher relaxation order $d$ to converge than the other three methods. When $s=2$, the complexity growth of the LME approach is the same as that of the SOS approach. Thus, despite its potentially faster convergence speed, the LME approach suffers from the same rapid complexity growth just as the dense SOS approach.   
In contrast, 
our CS-LME approach leads to SDP problems of a scale comparable with that of CS-SOS, and thus enjoys a less aggressive complexity growth. 
Meanwhile, it is expected to converge faster than CS-SOS as it incorporates the first-order optimality condition in the relaxation just as the LME approach, as shown in \Cref{sc:ne}.

In the above, we compared the maximal PSD matrix size in the SDP problems arising from different relaxation approaches when $s=2$.
To examine the number and size  of  all the PSD matrices in the SDP problems, one needs, in addition, the structure information of the functions $(f, g, h)$.  
The next example compares the SDP problem scale in detail for a box-constrained problem with a quadratic objective function. 
\begin{example}\label{ep:2-block}
Let $N$ and $k$ be positive integers and
\[\cI_1=\left\{1\ddd N\right\},\quad \cI_2=\left\{N+1-k\ddd 2N-k\right\}.\]
Note that this is a special two-block case  with $n_1=n_2=N$. 
Consider problem~\eqref{eq:cs-polyopt} with this csp $(\cI_1,\cI_2)$
and box constraints
\be\label{eq:gex3}\begin{array}{c}
\displaystyle g^{(1)}=(x_1,\,1-x_1\ddd \, x_{N-k},\,1-x_{N-k}),\\
\displaystyle g^{(2)}=(x_{N+1-k},\,1-x_{N+1-k}\ddd \, x_{2N-k},\,1-x_{2N-k}).
\end{array}\ee
The LMEs and CS-LMEs can be similarly given as in (\ref{eq:boxlme}) and (\ref{eq:box-cslme}), respectively, and we omit explicit expressions of them for the cleanness of this paper. 
Let $f_1$ and $f_2$ be quadratic functions. We present in~\Cref{tab:call} the number and size of all the PSD matrices in the four different approaches.
\begin{table}[htb]
\centering
\caption{Size and number of PSD matrices in the $d$th order relaxation of the four methods for the box constrained problem~\eqref{eq:gex3} with quadratic objective functions. }
\begin{tabular}{c|l}  \hline
{Relaxation }  &  Size and number of PSD matrices size  \\
approach &  in the $d$th order relaxation\\
 \hline 
 \multirow{2}{*}{SOS} &   one PSD matrix of size $\textstyle\binom{2N-k+d}{d} \times \binom{2N-k+d}{d}$, \\
  & $4N-2k$ PSD matrices of size  
  $\textstyle\binom{2N-k+d-1}{d-1} \times \binom{2N-k+d-1}{d-1}$.\\
 \hline 
 \multirow{2}{*}{CS-SOS} &  two PSD matrices of size $\textstyle\binom{N+d}{d}\times  \binom{N+d}{d}$,   \\
 &  $4N-2k$ PSD matrices of size $\textstyle\binom{N+d-1}{d-1}\times  \binom{N+d-1}{d-1}$.\\
 \hline 
 \multirow{2}{*}{LME} & one PSD matrix of size  $\textstyle\binom{2N-k+d}{d} \times  \binom{2N-k+d}{d} $, \\ 
 & $8N-4k$ PSD matrices of size  $\textstyle\binom{2N-k+d-1}{d-1} \times  \binom{2N-k+d}{d} $. \\
 \hline 
 \multirow{2}{*}{CS-LME} &   two PSD matrices of size $\textstyle\binom{N+k+d}{d} \times \binom{N+k+d}{d} $, \\ 
 &  $8N-4k$ PSD matrices of size 
 $\textstyle\binom{N+k+d-1}{d-1} \times \binom{N+k+d-1}{d-1}. $
\end{tabular}\label{tab:call}
\end{table}
\Cref{tab:call1010} is an instantiation of the numbers in~\Cref{tab:call} for the special case when $N=10$ and $k=2$. 
\begin{table}[htb]
\centering
\caption{Instantiation of~\Cref{tab:call} when $N=10$ and $k=2$. For example, the bottom-right block reads as follows: the $4$th order relaxation of the CS-LME approach corresponds to an SDP problem with two 1820-by-1820 PSD matrices and seventy-two 455-by-455 PSD matrices.   }
\begin{tabular}{c|c|c|c}  \hline
Relaxation  &   \multirow{2}{*}{ $d=2$} & \multirow{2}{*}{$d=3$} & \multirow{2}{*}{$d=4$} \\
approach & & &\\
 \hline 
 {SOS} &  $\left(1, 190 \right)$, $\left(36, 19\right)$ 
 &  $\left(1, 1330 \right)$, $\left(36, 190 \right)$ &  $\left(1, 7315 \right)$, $\left(36, 1330 \right)$
\\
 \hline 
{CS-SOS} &  $\left(2, 66 \right)$,  $\left(36, 11 \right)$ &  $\left(2, 286 \right)$, $\left(36, 66 \right)$ & $\left(2, 1001 \right)$, $\left(36, 286 \right)$  \\
 \hline 
 {LME} & $\left(1, 190 \right)$, $\left(72, 19 \right)$  & $\left(1, 1330 \right)$, $\left(72, 190 \right)$ &$\left(1, 7315\right)$, $\left(72, 1330 \right)$\\ 
 \hline 
 {CS-LME} &   $\left(2, 91\right)$, $\left(72, 13 \right)$ &   $\left(2, 455 \right)$, $\left(72, 91 \right)$   &   $\left(2, 1820 \right)$, $\left(72, 455 \right)$  
\end{tabular}\label{tab:call1010}
\end{table}
\end{example}

Now we consider general multiblock cases. 
If there exists a common variable in all the blocks, i.e., if there is some $j\in[n]$ such that $j\in\cI_i$ for all $i\in[s]$ (e.g.,  $s=2$ or~\Cref{ex:5balls}),
then the LME reformulation does not have correlative sparsity. 
In this case, the SDP problem scale of the LME approach grows similarly to that of the dense SOS relaxations. 
However, in general,
though the LME reformulation usually breaks the csp of the original problem,
it may  have a weaker correlative sparsity.
The following example is such an exposition.

\begin{example}
\label{ep:cspep}
Let $N>k$ be two positive integers.
Consider the following csp
\begin{equation}\label{eq:ex4}
\mc{I}_i=\left\{(N-k)(i-1)+1,\ldots,(N-k)(i-1)+N\right\},\enspace \forall i=1,\ldots,s.
\end{equation}
When $N=3$ and $k=2$, it corresponds to the csp of the Broyden tridiagonal function~\cite[Example 3.4]{Lasserre06}.
The {directed clique} tree $(V,A)$ associated to the sparsity pattern~\eqref{eq:ex4} is given by
$$
A=\{(i, {i-1}): i=2,\ldots,s\}.
$$
For each arc $(i, {i-1})\in A$, the set of joint indices is
$$
\cC_{i,i-1}=\cI_i \cap \cI_{i-1}=\{(N-k)(i-1)+1,\ldots,(N-k)(i-2)+N\}.
$$
Note that $|\cI_i|=N$ and $|\cC_{i,i-1}|=k$ for each $i\in [s]$.
The auxiliary variables are
\be\label{eq:chainnu}
\bigcup_{i=2}^s \bigcup_{j=1}^k \big\{
\nu_{i,i-1,(N-k)(i-1)+j}
\big\}.
\ee

\begin{table}[ht!]
\centering
\caption{The maximal PSD matrix size in $d$th order relaxation of the four  methods  when the csp is given by~\eqref{eq:ex4}. }
\begin{tabular}{c|l}  \hline
Relaxation  &  \multirow{2}{*}{ Maximal PSD matrix size in $d$th order relaxation }\\ 
approach & \\
 \hline 
 SOS  &  $\textstyle\binom{(N-k)(s-1)+N+d}{d} \times \binom{(N-k)(s-1)+N+d}{d}$ \\
 \hline 
 CS-SOS &  $\textstyle\binom{N+d}{d}\times  \binom{N+d}{d}$ \\
 \hline 
 LME &  $\textstyle\binom{(N-k)\left\lfloor \frac{N-1}{N-k} \right\rfloor+N+d}{d} \times \binom{(N-k)\left\lfloor \frac{N-1}{N-k} \right\rfloor+N+d}{d} $ \\
 \hline 
 CS-LME &  $\textstyle\binom{2k+N+d}{d} \times  \binom{2k+N+d}{d} $ \\
\end{tabular}\label{tab:cNks}
\end{table}
For the sparsity pattern~\eqref{eq:ex4}, the maximal clique size in the csp graph of the CS-LME reformulation~{\eqref{eq:pblmesparse}} is 
$$
N+2k.
$$  
In contrast, 
the maximal clique size in the original LME reformulation~\eqref{eq:pblme} is 
$$
(N-k)\left\lfloor \frac{N-1}{N-k} \right\rfloor+N . 
$$
We give in~\Cref{tab:cNks} the maximal PSD matrix size of the four methods for solving~\eqref{eq:cs-polyopt} with csp given by~\eqref{eq:ex4}.
\Cref{tab:cNks} shows that the SDP problem scale of CS-LME  is significantly smaller than SOS and LME  when  
$N\gg k$. Recall that $N$ is the size of the blocks while $k$ is the number of overlapping variables between two successive blocks. Thus $N/k$ can be seen as a measure of the partial separability of the problem. We speculate that the larger $N/k$ is,  the more efficient the CS-LME approach is compared with the other three approaches\footnote{The overall performance depends on both the SDP problem scale and the convergence rate with respect to the relaxation order $d$.  }.
See~\Cref{ep:uncons} for a numerical evidence with $N=15$,  $k=2$, and $s=10$. 
\end{example}
\begin{remark}\label{rem:csnl}To end this section, we would like to point out that for small-scale problems, the LME approach has outstanding performance, especially when  the  SOS approach  cannot find the global minimum with a low relaxation order; see~\cite{nie2019tight}.
For small-scale problems with csp, the LME approach may still be faster than the CS-LME approach because the latter needs to add auxiliary variables to maintain the csp. See~\Cref{ep:box-ne} for a numerical example of a small-scale problem. 

In general, we expect CS-LME to perform better than the other three approaches when the cliques in the csp graph  of the CS-LME reformulation are not much larger than that of the LME reformulation.  Since 
$
|\hat \cI_i |=|\cI_i|+ |\cJ_i|
$,  this occurs when 
\begin{enumerate}
\item  the number of overlapping variables between any two blocks $\cI_i$ and $\cI_j$ is small;
\item  each node in the directed clique tree $G(V,A)$ returned by~\Cref{A:TC}  has a small number of children. 
\end{enumerate}
These two conditions ensure that only a small number of auxiliary variables  $|\cJ_i|$ must be added to each block.  
\end{remark}

\section{Numerical experiments}
\label{sc:ne}
In this section,
we present numerical experiments that apply CS-LMEs to solve polynomial optimization problems with a given csp.
We directly call the software {\tt TSSOS} \footnote{\url{https://github.com/wangjie212/TSSOS}} \cite{TSSOS,CSTSSOS} to solve the CS-TSSOS relaxation of the CS-LME reformulation~\eqref{eq:pblmesparse}.
Note that CS-TSSOS relaxation exploits both correlative and term sparsity in the polynomial optimization problem.
{As recalled in Section~\ref{sc:cssos},
the convergence of CS-TSSOS is guaranteed when the CS-SOS relaxation is convergent (with option {\tt TS={"block"}}).}
The software {\tt Mosek} is applied to solve the SDPs with default settings.
The computation was implemented on {a} Lenovo x1 Yoga laptop,
with an Intel Core i7-1185G7 CPU at 3.00GHz$\times $4 cores and 16GB of RAM,
on the Windows 11 operating system.

For all polynomial optimization problems in this section,
we compare the performance of several approaches.
First, we solve the problem directly by {\tt CS-TSSOS} with options {\tt TS={"block"}} and {\tt TS={"MD"}}, respectively (see \cite{TSSOS} for more details).
Then, we solve the LME reformulation~\eqref{eq:pblme} introduced in \cite{nie2019tight} when it exists.
Note that when original LMEs are applied, correlative sparsity for the reformulation is usually corrupted.
Last, we solve the CS-LME reformulation~\eqref{eq:pblmesparse}.
For both LME reformulation (\ref{eq:pblme}) and CS-LME reformulation (\ref{eq:pblmesparse}),
the {\tt CS-TSSOS} is called with the option {\tt TS={"MD"}}.
Besides that, we use MATLAB software {\tt Gloptipoly 3} \cite{henrion2009gloptipoly} to implement dense relaxations with {\tt Mosek} being applied to solve the SDPs.
{We say a relaxation \textit{``fail to solve''} when we cannot get a sensible optimal value for it.
This is the case when we suspect SDP is unbounded as {\tt Mosek} reaches a negative objective value with a huge absolute value ($<-10^6$).}
\begin{example}
\label{ep:box-ne}
(i) Consider the polynomial optimization problem (\ref{eq:box}) in \Cref{ep:box}.
As mentioned in \Cref{ep:box}, its global minimum equals $0$.
The CS-LMEs for this problem are given by (\ref{eq:box-cslme}),
and the CS-LME reformulation is (\ref{eq:box-cslme-reform}).
{One may check that the archimedean condition is not satisfied by $\IQ_{\cI_1}(h^{(1)},g^{(1)})$.}
Besides that, the LME is given by (\ref{eq:boxlme}).
Numerical results for solving this problem are presented in \Cref{tab:simp}.
In the table, 
``$d$'' means the relaxation order,
``$l$'' represents the term sparsity level.
The columns ``no LME+{\tt block}'' and ``no LME+{\tt MD}'' are numerical results of applying {\tt CS-TSSOS} directly to the polynomial optimization problem with {\tt TS={"block"}} and {\tt TS={"MD"}} respectively,
the column ``LME'' corresponds to solving the  LME reformulation,
and the column ``CS-LME'' represents the relaxation results of the CS-LME reformulation.
The ``error'' is the absolute value of the difference of optimal value for this polynomial optimization problem and the approximation computed by the semidefinite relaxation,
and ``time'' is the time consumption in seconds for computing this approximation.
When a superscript $^*$ is marked, it means this lower bound was computed with the highest level of term sparsity within the current relaxation order.

From the table, one can see that when there were no LMEs exploited,
{\tt CS-TSSOS} could not get an approximation for the global minimum of this problem with high accuracy (say, the error is less than $10^{-6}$).
Particularly, when $d=3$, the computed optimal values for both ``no LME+{\tt block}'' and ``no LME+{\tt MD}'' are less than $-10^{13}$, and we marked ``fail to solve'' in the table.
{Besides that, when $d=3$, {\tt Gloptipoly 3} failed to solve the problem (unboundedness suspected),
and obtained an approximated value with error equaling $3\cdot 10^{-9}$ in 0.50 second when $d=4$.}
In contrast, the LME approach took around 0.23 second to get the approximated global minimum,
and the CS-LME approach obtained the approximated minimum in 0.53 second.

(ii) For the polynomial optimization problem in~\Cref{ep:boxexplain},
if we keep the objective function and the csp, but change the constraints to
\[\begin{array}{l}g^{(1)}(x^{(1)})=\left(1-{x^{(1)}}^Tx^{(1)},\, x^{(1)}_1,\, x^{(1)}_2\right),\quad  
  g^{(2)}(x^{(2)})=\left(1-{x^{(2)}}^Tx^{(2)},\, x^{(2)}_1,\, x^{(2)}_2\right)\end{array},\]
then the CS-LME becomes
\[\begin{array}{c}
\displaystyle \lmd^{(1)}_1 = -\frac{1}{2}{{x^{(1)}}^{\top}F^{(1)}},\quad \lmd^{(1)}_2 = F^{(1)}_1+2x^{(1)}_1\lmd^{(1)}_1,\quad \lmd^{(1)}_3 = F^{(1)}_2+2x^{(1)}_2\lmd^{(1)}_1, \\
\displaystyle \lmd^{(2)}_1 = -\frac{1}{2}{{x^{(2)}}^{\top}F^{(2)}},\quad \lmd^{(2)}_2 = F^{(2)}_1+2x^{(2)}_1\lmd^{(2)}_1,\quad \lmd^{(2)}_3 = F^{(2)}_2+2x^{(2)}_2\lmd^{(2)}_1.
\end{array}\]
However, one may check this problem does not have LMEs.
\begin{table}[ht!]
\small
\centering
\caption{Numerical results for~\Cref{ep:box-ne}(i)}
\begin{tabular}{cc|cc|cc|cc|cc}  \hline
\multirow{2}{*}{$d$} & \multirow{2}{*}{$l$}  & \multicolumn{2}{c|}{no LME+{\tt block}} & \multicolumn{2}{c|}{no LME+{\tt MD}} & \multicolumn{2}{c|}{LME} & \multicolumn{2}{c}{CS-LME}\\ \cline{3-10}
& & error & time &  error & time &  error & time &  error & time \\ \hline
3 & 1 & \multicolumn{2}{c|}{fail to solve} & \multicolumn{2}{c|}{fail to solve} & \multicolumn{2}{c|}{not defined} & \multicolumn{2}{c}{not defined}\\
3 & 2 & \multicolumn{2}{c|}{$^*$fail to solve} & \multicolumn{2}{c|}{fail to solve} &  &  & \\
3 & 3 & & & \multicolumn{2}{c|}{fail to solve} &  &  & \\
3 & 4 & & & \multicolumn{2}{c|}{$^*$fail to solve} &  &  & \\\hline
4 & 1 & $0.0134$ & 0.06s & $0.0437$ & 0.03s & $2\cdot 10^{-8}$  & 0.23s & $0.0014$  & 0.36s\\
4 & 2 & $^*0.0134$ & 0.07s & $0.0437$ & 0.03s &  &  & $\mathbf{1\cdot 10^{-7}}$ & {\bf 0.53s} \\
4 & 3 &  &  & $0.0140$ & 0.13s & && \\
\vdots &\vdots & \vdots &\vdots &\vdots &\vdots &   &&&\\
10 & 1 & $0.0038$ & 25.76s & $0.0337$ & 8.26s&   &&&\\
10 & 2 & $0.0038$ & 74.82s & $0.0152$ & 11.18s&   &&&\\
\end{tabular}
\label{tab:simp}
\end{table}

With the new constraints, one can similarly check that the global minimum is still $0$.
Numerical results for solving this problem are presented in \cref{tab:ep:box-ne2},
where symbols and notation are similarly defined as in~\Cref{tab:simp}.
From the table,
one can see that without CS-LMEs,
{\tt CS-TSSOS} cannot find the global minimum with satisfying error in 61 seconds for the option {\tt TS="block"}, and in 78 seconds for the option {\tt TS="MD"}.
Besides that, {\tt Gloptipoly} got the lower bound $-2\cdot 10^{-5}$ in 0.30 second for $d=3$,
and got $-5\cdot 10^{-9}$ in 0.52 second for $d=4$.
For the CS-LME approach, we obtained an approximation $-9\cdot 10^{-7}$ for the global minimum in $1.69$ seconds.
\begin{table}[htb]
\small
\renewcommand{\arraystretch}{1}
\centering
\caption{Numerical results for~\Cref{ep:box-ne}(ii)}
\label{tab:ep:box-ne2}
\begin{tabular}{cc|cc|cc|cc}  \hline
\multirow{2}{*}{$d$} & \multirow{2}{*}{$l$} &  \multicolumn{2}{c|}{no LME+{\tt block}} &  \multicolumn{2}{c|}{no LME+{\tt MD}} & \multicolumn{2}{c}{CS-LME}\\ \cline{3-8}
& & error & time & error & time & error & time \\ \hline
3 & 1 & $0.0146$ & 0.02s & $0.0531$ & 0.01s  &  \multicolumn{2}{c}{not defined} \\ 
3 & 2 & $^*0.0140$ & 0.02s & $0.0480$ & 0.01s & & \\ \hline
4 & 1 & $0.0074$ & 0.04s & $0.0495$ & 0.04s  &  $0.0018$ & 0.41s \\ 
4 & 2 & $^*0.0070$ & 0.06s & $0.0450$ & 0.04s & $0.0016$ & 0.42s \\ \hline
5 & 1 & $0.0045$ & 0.15s & $0.0492$ & 0.14s  &  $2\cdot10^{-5}$ & 0.98s \\ 
5 & 2 & $^*0.0044$ & 0.28s & $0.0448$ & 0.16s & $\mathbf{9\cdot10^{-7}}$ & \highlight{1.69s} \\ \hline
10 & 1 & $0.0049$ & 6.45s & $0.0437$ & 13.82s  &   &  \\ 
10 & 2 & $^*0.0034$ & 61.41s& $0.0245$ & 12.99s &  &  \\
10 & 3 & 	 & 	& $0.0035$ & 59.52s &  &  \\ 
\vdots & \vdots & & & \vdots & \vdots & & \\  
10 & 8 & 	 & 	& $^*0.0034$ & 78.16s &  &  \\ 
\end{tabular}
\end{table}
\end{example}

For all remaining examples in this section,
symbols and notation in tables are similarly defined as in~\Cref{tab:simp},
and we shall not repeat explaining them,
for the neatness of this paper.

\begin{example}\label{ex:5balls}
{Consider the csp given in~\Cref{ex:cscser}.
For each $i=1\ddd 5$,
we let $f_i(x^{(i)})$ be the Choi-Lam's form
\[f_i(x^{(i)}) = (x^{(i)}_1x^{(i)}_2)^2+(x^{(i)}_1x^{(i)}_3)^2+(x^{(i)}_2x^{(i)}_3)^2+{x^{(i)}_4}^4-4x^{(i)}_1x^{(i)}_2x^{(i)}_3x^{(i)}_4,\]
and let
\[ g^{(i)} = (1-{x^{(i)}}^{\top}x^{(i)}),\quad  h^{(i)}=\emptyset.\]
Again, by the inequality of arithmetic and geometric means,
all $f_i$ are nonnegative, and $f_i(x^{(i)})=0$ when $x^{(i)}_1=\dots=x^{(i)}_4$.
Thus we know the optimal value for minimizing $f_1(x^{(1)})+\dots+f_5(x^{(5)})$ over the set given by $g^{(i)}(x^{(i)})\ge0$ for all $i=1\ddd 5$ is $0$.
For this problem,
the CS-LMEs can be given as 
\[\lmd^{(i)}=-\frac{{x^{(i)}}^{\top} F^{(i)} }{2}.\]
However, there do not exist LMEs, which can be similarly shown as in~\Cref{ep:ball}.
Numerical results of solving this problem using CS-TSSOS directly, the LME approach, and the CS-LME approach are presented in \Cref{tab:5balls}.

From the table, one can see that without CS-LMEs,
{\tt CS-TSSOS} cannot find the global minimum with the option {\tt TS="MD"} (interestingly, it returned the same lower bound $-0.1709$ for all $d=2\ddd 15$),
and cannot get an approximation for the global minimum with an error less than $0.0001$ in 6807 seconds with {\tt TS="block"}.
Moreover, {\tt Gloptipoly 3} obtained the lower bound $-0.1709$ when $d=2$ using 0.99 second,
and obtained the lower bound $-0.0135$ in 346.42 seconds when $d=3$.
In contrast, the CS-LME approach took 11.35 seconds to obtain an approximated minimum with an error equal to $6\cdot 10^{-6}$,
and took 107.62 seconds to obtain an approximated minimum with an error equal to $3\cdot 10^{-8}$.

\begin{table}[ht!]
\small
\renewcommand{\arraystretch}{1}
\centering
\caption{Numerical results for~\Cref{ex:5balls}}
\begin{tabular}{cc|cc|cc|cc}  \hline
\multirow{2}{*}{$d$} & \multirow{2}{*}{$l$} &  \multicolumn{2}{c|}{no LME+{\tt block}} &  \multicolumn{2}{c|}{no LME+{\tt MD}} & \multicolumn{2}{c}{CS-LME}\\ \cline{3-8}
& & error & time & error & time & error & time \\ \hline
2 & 1  & $^*0.0531$ & 0.01s & $^*0.1709$ & 0.01s &  \multicolumn{2}{c}{not defined} \\ \hline
3 & 1  & $^*0.0480$ & 0.01s & $^*0.1709$ & 0.02s & $0.0080$ & 5.82s \\ \hline
4 & 1  & $^*0.0495$ & 0.04s & $^*0.1709$ & 0.05s & $1\cdot 10^{-5}$ & 6.55s \\ 
4 & 2  & 	 & 	 & 	 & 	 & $6\cdot 10^{-6}$ & 11.35s \\ \hline
5 & 1  & $^*0.0450$ & 0.04s & $^*0.1709$ & 0.16s & $\mathbf{3\cdot10^{-8}}$ & {\bf 107.62s} \\ 
\vdots & \vdots & \vdots & \vdots & \vdots & \vdots & & \\
15 & 1  & $^*0.0001$ & 6807.45s & $^*0.1709$ & 554.09s & &  \\  
\end{tabular}
\label{tab:5balls}
\end{table}
}
\end{example}
\begin{example}\label{ex:box10}
Consider the box-constrained problem in \Cref{ep:2-block}.
Let $n_1=n_2=10,\ k=2$, and let $(i=1,2)$
\[f_i(x^{(i)})=\left(\sum\nolimits_{j=1}^{10}x^{(i)}_j+1\right)^2-4\left(\sum\nolimits_{j=1}^{
9}x^{(i)}_jx^{(i)}_{j+1}+x^{(i)}_1+x^{(i)}_{10} \right).\]
The LMEs and CS-LMEs can be similarly given by (\ref{eq:boxlme}) and (\ref{eq:box-cslme}), respectively.
One may check that the archimedean condition is not satisfied by $\IQ_{\cI_1}(h^{(1)},g^{(1)})$.
Furthermore, for $d=2\ddd 3$, the structure of SDPs obtained by the dense relaxation, CS-SOS relaxations, the LME approach, and the CS-LME approach are given in \Cref{tab:call1010}.

The minimum for this problem is achieved at the KKT point $(1,0,\dots,0,1)$, which equals $0$ (see also \cite{nie2019tight}).
This can also be numerically certified by {\tt Gloptipoly 3} via the {\it flat truncation} \cite{nie2013certifying}.
Indeed, {\tt Gloptipoly 3} got an approximation to the global minimum $-3\cdot10^{-8}$ in 26.25 seconds.
Numerical results of solving this problem using CS-TSSOS directly, the LME approach and the CS-LME approach are presented in \Cref{tab:box10}. From the table, one can see that without LMEs,
{\tt CS-TSSOS} could not find  an approximation for the global minimum with a desired accuracy when ${\tt TS=``MD"}$ within 11.51 seconds,
and took 36.67 seconds to get the minimum when ${\tt TS=``block"}$.
The LME approach took 15.79 seconds to get the approximation with the desired accuracy.
In contrast, the CS-LME approach only took 2.46 seconds to get an approximated global minimum with the error equal to $4\cdot 10^{-7}$.
\begin{table}[ht!]
\small
\centering
\caption{Numerical results for~\Cref{ex:box10}}
\begin{tabular}{cc|cc|cc|cc|cc}  \hline
\multirow{2}{*}{$d$} & \multirow{2}{*}{$l$}  & \multicolumn{2}{c|}{no LME+{\tt block}} & \multicolumn{2}{c|}{no LME+{\tt MD}} & \multicolumn{2}{c|}{LME} & \multicolumn{2}{c}{CS-LME}\\ \cline{3-10}
& & error & time &  error & time &  error & time &  error & time \\ \hline
2 & 1 & $^*0.0067$ & 1.08s & $0.0739$ & 0.10s & $^*1\cdot 10^{-7}$ & 15.79s & $\mathbf{^*4\cdot10^{-7}}$ & \highlight{2.46s}\\\hline
3 & 1 & $9\cdot 10^{-9}$ & 36.67s & $^*0.0558$ & 0.78s &  &  &  & \\\hline
4 & 1 &  &  & $^*0.0105$ & 11.51s  &  &  &  & \\
\end{tabular}
\label{tab:box10}
\end{table}
\end{example}

\begin{example}
\label{ep:2blocks}
Let $s=2$ and
\[\cI_1=\{1,2,3,7\},\quad \cI_2=\{4,5,6,7\}.\]
Consider the polynomial optimization problem (\ref{eq:cs-polyopt}) with csp $\{\cI_1, \cI_2\}$,
where 
\[\begin{array}{c} \displaystyle
f_1(x^{(1)}) = x_1^4x_2^2+x_2^4x_3^2+x_3^4x_1^2-3(x_1x_2x_3)^2+x_2^2 + x_7^2(x_1^2+x_2^2+x_3^2),\\
f_2(x^{(2)}) = x_4x_5(10-x_6)+x_7^2(x_4+2x_5+3x_6);\\
g^{(1)}_1(x^{(1)}) = x_1-x_2x_3,\enspace g^{(1)}_2(x^{(1)}) = -x_2+x_3^2,\\
g^{(2)}_1(x^{(2)}) = 1-x_4-x_5-x_6,\enspace g^{(2)}_2(x^{(2)}) = x_4,\enspace g^{(2)}_3(x^{(2)}) = x_5,\enspace g^{(2)}_4(x^{(2)}) = x_6.
\end{array}\]
{Since $x_1^4x_2^2+x_2^4x_3^2+x_3^4x_1^2\ge 3(x_1x_2x_3)^2$ by the inequality of arithmetic and geometric means,
we have $f_1(x^{(1)})\ge0$ with the equality holds when $x_1=x_2=x_3=x_7=0$.
On the other hand, $f_2$ is nonnegative on the feasible set given by $g^{(2)}(x^{(2)})\ge0$, and $f_2(x^{(2)})=0$ when $x_4x_5=0$ and $x_7=0$.
So, the global minimum for this problem is $0$, which is attain at $(0,0,0,t,0,0,0)$ and $(0,0,0,0,t,0,0)$ for all $t\in[0,1]$.
Also, one may check that this problem has an unbounded feasible set.}
For this problem, let
\[F^{(1)}=\nabla_1 f_1+\nu_{2,1,7}\ei_4,\quad 
F^{(2)}=\nabla_2 f_2+\nu_{2,1,7}\ei_4,\]
then the CS-LMEs are
\[\begin{array}{c} \displaystyle
\lambda^{(1)}_1 = F^{(1)}_{1},\quad 
\lambda^{(1)}_2 = [-x_3, -1, 0, 0]\cdot F^{(1)},\\
\lambda^{(2)}_1= - x_{4:6}^{\top} F^{(2)}_{1:3},\quad 
\lambda^{(2)}_2= F^{(2)}_{1}+\lambda^{(2)}_1,\quad 
\lambda^{(2)}_3= F^{(2)}_{2}+\lambda^{(2)}_1,\quad 
\lambda^{(2)}_4= F^{(2)}_{3}+\lambda^{(2)}_1.
\end{array}\]

The numerical results for solving this problem are presented in~\Cref{tab:2blocks}.
From the table, one can see that when there were no LMEs exploited,
{\tt CS-TSSOS} could not get an approximation for the global minimum of this problem with an error less than $0.0001$ within 271.95 seconds,
while the original LME approach took around 84.13 seconds to get the approximated value with an error equaling $2\cdot10^{-7}$.
{Moreover, when $d=3$ and $4$, {\tt Gloptipoly 3} failed to solve the problem (unboundedness suspected),
and it took 2264 seconds to get the lower bound $-120.82$ when $d=5$.}
In contrast, the CS-LME approach obtained an approximated minimum whose error was $9\cdot10^{-8}$ in $18.54$ seconds.
\begin{table}[htb]
\small
\renewcommand{\arraystretch}{1}
\centering
 \caption{Numerical results for~\Cref{ep:2blocks}}
\begin{tabular}{cc|cc|cc|cc|cc}  \hline
\multirow{2}{*}{$d$} & \multirow{2}{*}{$l$} &  \multicolumn{2}{c|}{no LME+{\tt block}} &  \multicolumn{2}{c|}{no LME+{\tt MD}} & \multicolumn{2}{c|}{LME} & \multicolumn{2}{c}{CS-LME}\\ \cline{3-10}
&   & error & time & error & time & error & time & error & time\\ \hline
3 & 1 & \multicolumn{2}{c|}{fail to solve} & \multicolumn{2}{c|}{fail to solve} & \multicolumn{2}{c|}{not defined} & \multicolumn{2}{c}{not defined} \\ 
3 & 2 & $^*>10^{8}$ & 0.18s & \multicolumn{2}{c|}{fail to solve} & \multicolumn{2}{c|}{not defined} & \multicolumn{2}{c}{not defined} \\ 
\vdots & \vdots & & & \vdots & \vdots & \vdots & \vdots & \vdots & \vdots  \\
3 & 5 & & & \multicolumn{2}{c|}{$^*$fail to solve} & \multicolumn{2}{c|}{not defined} & \multicolumn{2}{c}{not defined} \\ \hline
4 & 1 & $>10^{7}$ & 0.47s & \multicolumn{2}{c|}{fail to solve}  &  $1519.49$ & 4.95s & $645.71$ & 0.77s\\
4 & 2 & $^*>10^{5}$ & 0.59s & $>10^6$ & 0.45s  & $35.36$ & 5.28s & $23.62$ & 0.94s\\ 
\vdots & \vdots & \vdots & \vdots & \vdots & \vdots & \vdots & \vdots & \vdots & \vdots  \\
5 & 2 & $^*265.61$ & 2.60s &  $>10^5$         & 1.44s & $2\cdot 10^{-7}$ & 84.13s & $0.0324$ & 5.42s \\ 
\vdots & \vdots & \vdots & \vdots & \vdots & \vdots &  &  & \vdots & \vdots  \\
6 & 1 & $18.19$ & 6.28 & $102.78$         & 5.59s & & & $\mathbf{9\cdot 10^{-8}}$ & \highlight{18.54s} \\
\vdots & \vdots & \vdots  & \vdots & \vdots  & \vdots && & \multicolumn{2}{c}{ }  \\
8 & 2 & $^*0.0001$ & 224.77s & $0.0079$ & 75.64s & && \\ 
\vdots & \vdots &    &  & \vdots  & \vdots && & \multicolumn{2}{c}{ }  \\
8 & 5 & & & $^*0.0002$ & 322.68s & && \\
\end{tabular}
\label{tab:2blocks}
\end{table}
\end{example}

\begin{example}
\label{ep:nonaive}
Let $s=5$ and
\[\begin{array}{c}
\cI_1=\{1,2,3,4,17,18,19\},\ 
\cI_2=\{5,6,7,8,18,19,20\},\ \\
\cI_3=\{9,10,18,19,20\},\ 
\cI_4=\{11,12,17,18\},\ 
\cI_5=\{13,14,15,16,17\}.\end{array}\]
Consider the polynomial optimization problem (\ref{eq:cs-polyopt}) with csp $(\cI_1,\cI_2\ddd \cI_5)$, where
\[\begin{aligned}
f_1(x^{(1)})\ =&\ (x_1-x_{17})^2+(x_2-x_{18})^2+(x_3-x_{19})^2+x_4^2x_{17},\\ 
f_2(x^{(2)})\ =&\ x_{18}^2+x_{19}^2+x_{20}^2-x_5(x_6+x_7+x_8),\\
f_3(x^{(3)})\ =&\ x_9x_{10}(20-x_{18}-x_{19}-x_{20}),\\ 
f_4(x^{(4)})\ =&\ (x_{11}-x_{17})^2+(x_{12}+x_{18}-1)^2,\\
f_5(x^{(5)})\ =&\ (x_{17}-x_{13}+x_{14})^2+x_{15}x_{16},\\
g^{(1)}(x^{(1)})\ =&\ \left({\bf1}_7-x^{(1)},\ x^{(1)}_{1:4}+{\bf1}_7\right),\\ 
g^{(2)}(x^{(2)})\ =&\ \left(3-x^{(2)}_1-2\sum\nolimits_{j=2}^4x^{(2)}_j-\sum\nolimits_{j=5}^{7}x^{(2)}_j,\ x^{(2)}_1,\dots,\ x^{(2)}_7 \right),\\
g^{(3)}(x^{(3)})\ =&\ \left(1-\sum\nolimits_{j=1}^5x^{(3)}_j,\ x^{(3)}_{1},\ x^{(3)}_{2} \right),\\ 
g^{(4)}(x^{(4)})\ =&\ 1-{x^{(4)}}^{\top}x^{(4)},\quad 
g^{(5)}(x^{(5)})\ =\ x^{(5)}.
\end{aligned}\]
It is clear that except $f^{(2)}$, all other $f^{(i)}$ are nonnegative over the set given by $(g^{(1)},g^{(2)}\ddd g^{(5)})$.
For $f^{(2)}$, its minimum $-\frac{9}{8}$ is attained at the KKT point $x^{(2)}=\left(\frac{3}{4},0,0,\frac{3}{2},0,0,0\right)$.
Indeed, one may check that the global minimum for this problem is $-\frac{9}{8}$.
For this problem, the set of edges is \[{A}=\{(2,1),\ (3,2),\ (4,1),\ (5,4)\}.\]
The auxiliary variables are
\[\nu_{2,1,18},\ \nu_{2,1,19},\ \nu_{3,2,18},\ \nu_{3,2,19},\ \nu_{3,2,20},\ \nu_{4,1,17},\ \nu_{4,1,18},\ \nu_{5,4,17}.\]
If we let $F^{(i)}$ be given as in (\ref{eq:Fi}),
then CS-LMEs are
\[\begin{array}{c}\displaystyle
\lmd^{(1)}_{1:4} = -\frac{1}{2}\cdot F^{(1)}_{1:4}\circ (\mathbf{1} + x_{1:4}),\quad
\lmd^{(1)}_{5:7} = - F^{(1)}_{5:7},\quad
\lmd^{(1)}_{8:11}= F^{(1)}_{1:4}+\lmd^{(1)}_{1:4};\\
\lmd^{(2)}_1= -\frac{1}{3} {F^{(2)}}^{\top} x^{(2)}, \quad
\lmd^{(2)}_{2:4} = 2\lmd^{(2)}_1+F^{(2)}_{1:3}, \quad
\lmd^{(2)}_{5:8} = 2\lmd^{(2)}_1+F^{(2)}_{4:7};\\
\lmd^{(3)}_1= - {F^{(3)}}^{\top} x^{(3)}, \quad
\lmd^{(3)}_{2:3} = \lmd^{(3)}_1+F^{(3)}_{1:2};\quad \lmd^{(4)}= -\frac{1}{2} {F^{(4)}}^{\top} x^{(4)}; \quad
\lmd^{(5)} = F^{(5)}.\end{array}\]
We would like to remark that the tuple $(g^{(1)},g^{(2)}\ddd g^{(5)})$ is singular,
so original LMEs do not exist.
The numerical results for solving this problem are presented in~\Cref{tab:nonaive}.
From the table, one can see that when there were no LMEs exploited,
{\tt CS-TSSOS} could not get an approximation for the global minimum with an error less than $0.001$ in 7697.33 seconds.
{Moreover, {\tt Gloptipoly 3} suspected unboundedness when $d=3$,
and the $4$th order dense relaxation cannot be solved due to the memory limit.}
In contrast, the CS-LME approach obtained an approximated minimum whose error was $1\cdot10^{-7}$ in $53.73$ seconds.
\begin{table}[htb]
\small
\renewcommand{\arraystretch}{1}
\centering
\caption{Numerical results for~\Cref{ep:nonaive}}
\begin{tabular}{cc|cc|cc|cc}  \hline
\multirow{2}{*}{$d$} & \multirow{2}{*}{$l$} &  \multicolumn{2}{c|}{no LME+{\tt block}} &  \multicolumn{2}{c|}{no LME+{\tt MD}} & \multicolumn{2}{c}{CS-LME}\\ \cline{3-8}
& & error & time & error & time & error & time \\ \hline
2 & 1 & \multicolumn{2}{c|}{fail to solve} & $^*>10^{6}$ & 0.28s & $9.5731$ & 1.32s \\ 
2 & 2 & $^*>10^6$ & $0.37 $ & \multicolumn{2}{c|}{ } & $0.3085$ & 1.50s \\  
\vdots & \vdots & \multicolumn{2}{c|}{ } & & & \vdots & \vdots  \\ 
2 & 5 & \multicolumn{2}{c|}{ } & \multicolumn{2}{c|}{ } & $^*0.1417$ & 10.05s \\ \hline 
3 & 1 & $1.6047$ & 3.59s & $1295.25$ & 0.71s  &  $\mathbf{4\cdot10^{-7}}$ & \highlight{60.41s} \\ 
3 & 2 &\multicolumn{2}{c|}{$^*$fail to solve}& $1276.92$ & 0.76s &  & \\
\vdots & \vdots & \vdots & \vdots & \vdots & \vdots &  \\
5 & 2 & $^*0.0069$ & 16663.91s & $9.3531$ & 252.95s &  \\
5 & 3 & \multicolumn{2}{c|}{ } & $0.0862$ & 7697.33s &  \\
\end{tabular}
\label{tab:nonaive}
\end{table}

\end{example}

{For the following two examples,
we do not run {\tt Gloptipoly 3} for solving them,
since the problem scales are too large for dense SOS relaxations.}
\begin{example}
\label{ep:uncons}
Consider the correlative sparsity pattern given in~\Cref{ep:cspep}.
Let $s=10$, $N=15$, and $k=2$.
For each $i\in [10]$,
let 
\[\begin{array}{c}
f_i(x)=\left({x^{(i)}}^Tx^{(i)}\right)^2-4\left((x^{(i)}_1x^{(i)}_2)^2+\dots+(x^{(i)}_4x^{(i)}_5)^2 +(x^{(i)}_5x^{(i)}_1)^2\right)\\
+\left( x^{(i)}_1+\dots+x^{(i)}_5-(x^{(i)}_{6:10})^{\top}x^{(i)}_{11:15} \right)^2.
\end{array}\]
Consider the unconstrained polynomial optimization problem
\begin{equation}\label{eq:uncons}
\begin{array}{lll}
&\min\limits_{x} & f_1(x^{(1)})+\dots+f_{10}(x^{(10)}).
\end{array}
\end{equation}
{For each $i\in[10]$,
the $\left({x^{(i)}}^Tx^{(i)}\right)^2-4\left((x^{(i)}_1x^{(i)}_2)^2+\dots+(x^{(i)}_4x^{(i)}_5)^2 +(x^{(i)}_5x^{(i)}_1)^2\right)$ is the {\it Horn's form} \cite{reznick2000some},
which is a nonnegative homogeneous polynomial.
Thus the global minimum of \eqref{eq:uncons} is $0$.}
For unconstrained problems, the system 
\[\phi^{(i)}(x^{(i)},\nu^{(i)})=0,\quad
\psi^{(i)}(x^{(i)},\nu^{(i)})\ge0,\quad \forall i\in[10]\]
reduces to 
\[
F^{(1)}(x^{(1)},\nu^{(1)}) = F^{(2)}(x^{(2)},\nu^{(2)}) = \cdots= F^{(10)}(x^{(10)},\nu^{(10)})=0,
\]
where every $F^{(i)}$ is given in (\ref{eq:Fi}) with auxiliary variables given in (\ref{eq:chainnu}).
Thus, the CS-LME typed reformulation (\ref{eq:pblmesparse}) becomes
\begin{equation}\label{eq:cs-uncons}
\left.\begin{array}{lll}
&\displaystyle \min ~~~ &f_1(x^{(1)})+\dots+f_{10}(x^{(10)})\\
&\enspace \mathrm{s.t.} \quad  &
F^{(1)}(x^{(1)},\nu^{(1)}) = F^{(2)}(x^{(2)},\nu^{(2)}) = \cdots= F^{(10)}(x^{(10)},\nu^{(10)})=0
\end{array}\right.
\end{equation}
Similarly, the original LME reformulation (\ref{eq:pblmesparse}) for (\ref{ep:ls}) becomes \begin{equation}\label{eq:lme-uncons}
\left.\begin{array}{lll}
&\displaystyle \min ~~~ &f_1(x^{(1)})+\dots+f_{10}(x^{(10)})\\
&\enspace \mathrm{s.t.} \quad  &
\nabla (f_1+f_2+\dots+f_{10})(x)=0
\end{array}\right.
\end{equation}

The numerical results for solving this problem are presented in~\Cref{tab:uncons}.
From the table, one can see that when there were no LMEs exploited,
{\tt CS-TSSOS} could not get a sensible approximation for the global minimum of this problem within 487.31 seconds,
while the original LME approach took around 270.40 seconds to get an approximated global minimum.
In contrast, the CS-LME approach obtained an approximated minimum whose error was $7\cdot10^{-10}$ in $20.48$ seconds.
\begin{table}[htb]
\small
\setlength{\tabcolsep}{5.5pt}
\renewcommand{\arraystretch}{1}
\centering
\caption{Numerical results for~\Cref{ep:uncons}}
\begin{tabular}{cc|cc|cc|cc|cc}  \hline
\multirow{2}{*}{$d$} & \multirow{2}{*}{$l$} &  \multicolumn{2}{c|}{no LME+{\tt block}} &  \multicolumn{2}{c|}{no LME+{\tt MD}} & \multicolumn{2}{c|}{LME} & \multicolumn{2}{c}{CS-LME}\\ \cline{3-10}
& & error & time & error & time & error & time & error & time \\ \hline
2 & 1 & \multicolumn{2}{c|}{$^*$fail to solve} & \multicolumn{2}{c|}{$^*$fail to solve} & \multicolumn{2}{c|}{fail to solve} & \multicolumn{2}{c}{fail to solve} \\ 
2 & 2 & \multicolumn{2}{c|}{ } & \multicolumn{2}{c|}{ } & \multicolumn{2}{c|}{$^*$fail to solve} & \multicolumn{2}{c}{$^*$fail to solve} \\ \hline
3 & 1 &  $^*>10^8$ & 78.43s  & $^*>10^8$ & 7.26s & $6\cdot 10^{-11}$ & 270.40s & $\mathbf{7\cdot 10^{-10}}$ & \highlight{20.48s} \\
4 & 1 &  \multicolumn{2}{c|}{$^*$out of memory} & $^*>10^6$ & 487.31s&  &  &  &  \\    
5 & 1 & & & \multicolumn{2}{c|}{$^*$out of memory}  && \\
\end{tabular}
\label{tab:uncons}
\end{table}
\end{example}

\begin{example}
\label{ep:ls}
In this example, we present numerical results by varying the number of blocks $s$.
For each $i\in [s]$, let $\mc{I}_i:=\{9i-8, 9i-7,\dots, 9i+1\}$.
Consider the following optimization problem
\begin{equation}\label{eq:ls} \left\{
\begin{array}{lll}
&\min\limits_{x} & f_1(x^{(1)})+f_2(x^{(2)})+\dots+f_s(x^{(s)})\\
&\st&  x^{(1)}_1\ge0,\ x^{(i)}_1+x^{(i)}_2+\dots+x^{(i)}_{10}\le 1,\ x^{(i)}_{2:10}\ge0,\enspace i\in [s]\\
\end{array} \right.
\end{equation}
In the above, 
\[f_i(x^{(i)})=\sum\nolimits_{j=1}^3x^{(i)}_{2j}x^{(i)}_{2j+1}+\left(\sum\nolimits_{j=7}^9(x^{(i)}_j)^3-3x^{(i)}_7x^{(i)}_8x^{(i)}_9\right)x^{(i)}_{10},\enspace i\in[s].\]
Since all variables are nonnegative,
by the inequality of arithmetic and geometric means,
each $f_i(x^{(i)})$ is nonnegative and reaches $0$ at $x^{(i)}={\bf 0}$,
and it is clear that (\ref{eq:ls}) has the csp $(\cI_1\ddd \cI_s)$ and its minimum value equals $0$.
Moreover, because for all $s\ge2$,
the matrix $\bG(x)$ given as in \eqref{eq:gsedf} does not have full column rank at $\ei_{10}$.
So \eqref{eq:ls} does not have LMEs.
For each $i\in[s-1]$,
we have the auxiliary variable $\nu_{i+1,i,9i+1}$.
Let $F^{(i)}$ be given in (\ref{eq:Fi}),
then CS-LMEs are
\[\begin{array}{c}\displaystyle
\lambda_1^{(1)} = -{F^{(1)}}^{\top}x^{(1)},\quad \lambda_{2:11}^{(1)}=F^{(1)}+\lambda_1^{(1)}; \\\lambda_1^{(i)} = -{F^{(i)}}^{\top}x^{(i)},\quad \lambda_{2:10}^{(i)}=F^{(i)}_{2:10}+\lambda_1^{(i)},\quad (i=2\ddd s). \end{array} \]

The numerical results for solving this problem with $s=2\ddd 7$ are presented in~\Cref{tab:ls}.
In the table, ``$s$'' represents the quantity $s$ in (\ref{eq:ls}),
and all other symbols and notations are similarly defined as in~\Cref{tab:simp} (see~\Cref{ep:box-ne}).
When $s=2$, one can see that when there were no CS-LMEs exploited,
{\tt CS-TSSOS} could not get an approximation for the global minimum of this problem with an error less than $0.01$ in 11366.94 seconds.
In contrast, the CS-LME approach obtained an approximated minimum whose error was $5\cdot10^{-9}$ in $1192.89$ seconds.
Moreover, when $s=3\ddd 7$, we do not present numerical results with relaxation order $d=3$ since we cannot get lower bounds that are close to $0$.
Also, results of approaches without CS-LMEs are not presented for $s\ge3$ and $d=4$,
because close lower bounds cannot be computed by these approaches with reasonable time consumption.
\begin{table}[htb]
\small
\renewcommand{\arraystretch}{1}
\centering
\caption{Numerical results for~\Cref{ep:ls}}
\begin{tabular}{c|cc|cc|cc|cc}  \hline
\multirow{2}{*}{$s$}& \multirow{2}{*}{$d$} & \multirow{2}{*}{$l$} &  \multicolumn{2}{c|}{no LME+{\tt block}} &  \multicolumn{2}{c|}{no LME+{\tt MD}} & \multicolumn{2}{c}{CS-LME}\\ \cline{4-9}
& & & error & time & error & time & error & time \\ \hline
\multirow{6}{*}{2} 
&  3 & 1 & $0.0735$ & 71.54s & $5369.40$ & 2.88s & $3.6067$ & 16.85s \\
&  3 & 2 & $^*0.0230$ & 196.25s & $ 624.22$ & 4.25s & $0.0680$& 35.02s \\
&  3 & 3 & & & $ 0.0238$ & 78.46s  & $0.0091$ & 353.71s \\
&  3 & 4 & & & $^*0.0230$ & 216.41s & $0.0071$ & 834.23s \\\cline{2-9}
&  4 & 1 & $0.0205$ & 11366.94s & $23.77$ & 682.21 & $\mathbf{5\cdot 10^{-9}}$ & \highlight{1192.89s} \\
&  4 & 2 & & & $0.0104$ & 71235.47 & \multicolumn{2}{c}{-} \\  
\hline
\multirow{1}{*}{3}
&  4 & 1 & &&&& $\mathbf{7\cdot 10^{-8}}$ & \highlight{1965.19s} \\
\hline
\multirow{1}{*}{4}
&  4 & 1 & &&&& $\mathbf{2\cdot 10^{-7}}$ & \highlight{2432.45s} \\ 
\hline
\multirow{1}{*}{5}
&  4 & 1 & &&&& $\mathbf{3\cdot 10^{-7}}$ & \highlight{2868.47s} \\
\hline
\multirow{1}{*}{6}
&  4 & 1 & &&&& $\mathbf{3\cdot 10^{-7}}$ & \highlight{4136.36s} \\
\hline
\multirow{1}{*}{7}
&  4 & 1 & &&&& $\mathbf{3\cdot 10^{-7}}$ & \highlight{4567.80s} \\ 
\end{tabular}
\label{tab:ls}
\end{table}
\end{example}

\section{Conclusions and discussions}
\label{sc:dis}
 We consider correlatively sparse polynomial optimization problems.
We introduce CS-LMEs to construct CS-LME reformations for polynomial optimization problems.
Under some general assumptions, we show that correlative SOS relaxations can get tighter lower bounds when solving the CS-LME reformulation instead of the original optimization problem.
Moreover, asymptotic convergence is guaranteed if the sequel of CS-SOS relaxations for the original polynomial optimization is convergent.
Numerical examples are presented to show the superiority of our new approach.

For future work, one wonders if the CS-SOS relaxation has finite convergence for solving CS-LME reformulations. Indeed, finite convergence for the original LME reformulation in \cite{nie2019tight} is guaranteed under mild conditions.
As demonstrated in Section~\ref{sc:ne}, the CS-LME approach usually finds the global minimum (up to a negligible numerical error) for polynomial optimization problems with a low relaxation order.
However, it is still open that if the finite convergence is guaranteed theoretically or not, even for generic cases.
Moreover, when the correlatively sparse polynomial optimization~\eqref{eq:cs-polyopt} is given by generic polynomials,
its KKT ideal is zero-dimensional. Thus the real variety given by equality constraints in~\eqref{eq:pblmesparse} is a finite set.
For the classical Moment-SOS relaxations,
finite convergence is theoretically guaranteed when equality constraints of the polynomial optimization give a zero-dimensional real variety,
as shown in \cite{NieReal}.
So, it is interesting to ask whether the analogous is true for CS-SOS relaxations. 
{Besides that, our numerical experiments indicate that the CS-LME approach can usually find the global minimum for polynomial optimization problems even if some $\IQ_{\cI^{(i)}} (h^{(i)},g^{(i)})$ is not archimedean. 
Therefore, an interesting question is whether the CS-LME approach has guaranteed asymptotic or finite convergence without the archimedean condition for every $\IQ_{\cI^{(i)}} (h^{(i)},g^{(i)})$.}

At last, we would like to remark that LMEs have broad applications in many polynomial-defined problems.
Therefore, a natural question is how to apply CS-LMEs to these applications.
For example, when a saddle point problem is given by polynomials with correlative sparsity, can we apply CS-LMEs to construct polynomial optimization reformulation similar to the one in \cite{nie2021saddle} for finding saddle points?

\appendix
\section{Computing LMEs and CS-LMEs}\label{sec:appLME}
{We introduce how to find LMEs and CS-LMEs for practical implementation.
As mentioned in \Cref{sc:preLMEs} and \Cref{sc:CSLME}, finding LMEs (resp., CS-LMEs) is equivalent to finding matrices of polynomials $L(x),\, D(x)$ (resp., $L^{(i)}(x),\, D^{(i)}(x)$) such that~\eqref{a:LGI} (resp.,~\eqref{eq:cs-LGI}) holds.
Note that the matrices $\bG(x)$ and $\bG^{(i)}(x)$ only depend on constraints,
and LMEs can be viewed as special cases of CS-LMEs that there only exists one block, i.e., $s=1$.
Here we only introduce how to get CS-LMEs, and the methodology for finding LMEs is similar.

Suppose the matrix of polynomial $\bG^{(i)}(x^{(i)})$ has full column rank over $\mathbb{C}^{n_i}$.
In general, (\ref{eq:cs-LGI}) gives a linear equation system.
Denote $\hm_i:=m_i+\ell_i$, and
\[L^{(i)}(x^{(i)}):=\left[\begin{array}{cccc}
L_{1,1}(x^{(i)}) & L_{1,2}(x^{(i)}) & \dots & L_{1,n_i}(x^{(i)}) \\
\vdots & \vdots & \vdots & \vdots  \\
L_{\hm_i,1}(x^{(i)}) & L_{\hm_i,2}(x^{(i)}) & \dots & L_{\hm_i,n_i}(x^{(i)})
\end{array}\right],
\]
\[D^{(i)}(x^{(i)}):=\left[\begin{array}{cccc}
D_{1,1}(x^{(i)}) & D_{1,2}(x^{(i)}) & \dots & D_{1,\hm_i}(x^{(i)}) \\
\vdots & \vdots & \vdots & \vdots  \\
D_{\hm_i,1}(x^{(i)}) & D_{\hm_i,2}(x^{(i)}) & \dots & D_{\hm_i,\hm_i}(x^{(i)})
\end{array}\right].
\]
Suppose all entries in $L^{(i)}(x^{(i)})$ and $D^{(i)}(x^{(i)})$ are polynomials whose degrees are not greater than $d$.
For each $j,k$, let (here for the $\alpha=(\alpha_1\ddd \alpha_{n_i})\in\mathbb{N}^{n_i}_d$, we denote ${x^{(i)}}^{\alpha}:={x^{(i)}_1}^{\alpha_1}{x^{(i)}_2}^{\alpha_2}\dots{x^{(i)}_{n_i}}^{\alpha_{n_i}}$)
\be\label{eq:LD}
L_{j,k}(x^{(i)})=\sum_{\alpha\in\mathbb{N}^{n_i}_d}L_{j,k,\alpha}\cdot{x^{(i)}}^{\alpha},\quad
  D_{j,k}(x^{(i)})=\sum_{\alpha\in\mathbb{N}^{n_i}_d}D_{j,k,\alpha}\cdot{x^{(i)}}^{\alpha}.\ee
Then (\ref{eq:cs-LGI}) can be written as the following linear equation system in variables $L_{j,k,\alpha}$ and $D_{j,k,\alpha}$:
\be\label{eq:lme-les}
\begin{aligned}
&\sum_{l=1}^{n_i}\left(\sum_{\alpha\in\mathbb{N}^{n_i}_d}L_{j,l,\alpha}\cdot{x^{(i)}}^{\alpha}\right)\cdot\frac{\pt c^{(i)}_{k}}{\pt x^{(i)}_l}(x^{(i)})
+ \left(\sum_{\alpha\in\mathbb{N}^{n_i}_d}D_{j,k,\alpha}\cdot{x^{(i)}}^{\alpha}\right)c^{(i)}_k(x^{(i)})\\
=\ &\left\{\begin{array}{c}
1\quad \mbox{if}\quad j=k,\\
0\quad \mbox{if}\quad j\ne k,
\end{array}\right.
\qquad (j\in[\hm_i],\ k\in[\hm_i]).
\end{aligned}
\ee
We remark that in \eqref{eq:lme-les}, the equality means that the polynomials on both sides are identically equaled.
By \cite[Proposition~5.2]{nie2019tight},
since $\bG^{(i)}(x)$ has full column rank over $\mathbb{C}^{n_i}$,
the system \eqref{eq:lme-les} must have solutions when $d$ is large enough.
Therefore, for each $i\in[s]$,
we solve the linear system \eqref{eq:lme-les} for solutions with a given degree $d$.
If we get a solution to \eqref{eq:lme-les}, then we recover polynomial matrices $L^{(i)}(x^{(i)})$ and $D^{(i)}(x^{(i)})$ (hence CS-LMEs) using this solution;
otherwise, we let $d\leftarrow d+1$ and solve \eqref{eq:lme-les} with the updated degree $d$, until a solution is obtained.

Sometimes, one may get CS-LMEs without actually computing polynomial matrices $L^{(i)}(x^{(i)})$ and $D^{(i)}(x^{(i)})$.
Instead, CS-LMEs can be directly obtained using the ``multiplication-cancellation'' trick
\footnote{This trick was introduced by Professor Jiawang Nie in his research group discussions. It is also mentioned in Section 6.3 of his new book {\it Moment and Polynomial Optimization} \cite{nie2023moment}.}.
This is shown in the following example.
\begin{example}\label{ex:sector-multi-cancel}
Consider the case that 
\[ g^{(i)}(x^{(i)})=\left(1-{x^{(i)}}^{\top}x^{(i)},\ x^{(i)}_1\ddd x^{(i)}_{n_i}\right).\]
Then the KKT-typed system (\ref{eq:sub-KKT}) for the $i$th block implies that
\begin{align}
\label{eq:optcond}&F^{(i)}(z^{(i)}) = -2\lmd^{(i)}_1\cdot x^{(i)}+\sum\nolimits_{j=1}^{n_i}\lmd^{(i)}_{j+1}\cdot e_j,\\
&\label{eq:complem}\lmd^{(i)}_1 \perp 1-{x^{(i)}}^{\top}x^{(i)},\quad 
\lmd^{(i)}_{j+1}\perp x^{(i)}_j\ (j\in[n_i]).\end{align}
By multiplying ${x^{(i)}}^{\top}$ on both sides of (\ref{eq:optcond}),
we get
\[{x^{(i)}}^{\top}F^{(i)}(z^{(i)}) = -2\lmd^{(i)}_1\cdot {x^{(i)}}^{\top}x^{(i)}+\sum\nolimits_{j=1}^{n_i}\lmd^{(i)}_{j+1}\cdot x^{(i)}_j.\]
Note that (\ref{eq:complem}) implies that $\lmd^{(i)}_1\cdot {x^{(i)}}^{\top}x^{(i)}=\lmd^{(i)}_1$
and $\lmd^{(i)}_{j+1}\cdot x^{(i)}_j=0$.
So we further have
\[{x^{(i)}}^{\top}F^{(i)}(z^{(i)}) = -2\lmd^{(i)}_1.\]
Therefore, again by (\ref{eq:optcond}), we get CS-LMEs that
\[\lmd^{(i)}_1 = -{x^{(i)}}^{\top}F^{(i)}(z^{(i)})/2,\quad
\lmd^{(i)}_{j+1} = F^{(i)}_{j}(z^{(i)})+2\lmd^{(i)}_1 \cdot x^{(i)}_j\ (j\in[n_i]).\]
\end{example}
We remark that though we do not get explicit expressions for $L^{(i)}(x^{(i)})$ and $D^{(i)}(x^{(i)})$, essentially, this trick is equivalent to finding solutions for (\ref{eq:cs-LGI}).
For instance, the step of multiplying ${x^{(i)}}^{\top}$ on both sides of (\ref{eq:optcond}) means that the first row of $L^{(i)}(x^{(i)})$ is ${x^{(i)}}^{\top}$.
Besides that, for some commonly used constraints (e.g., box, ball, simplex, etc.), LMEs are explicitly given in \cite{nie2019tight},
and they can be similarly applied to the construction of CS-LMEs.}

\section*{Acknowledgments}
The authors would like to thank the editor and anonymous reviewers for all their valuable comments and suggestions, which led to an improvement of the manuscript. We also thank Jiawang Nie and Jie Wang for their inspiring and helpful comments. Zheng Qu was partially supported by NSFC Young Scientist Fund grant 12001458 and Hong Kong Research Grants Council General Research Fund grant 17317122.
Xindong Tang was partially supported by the Start-up Fund P0038976/BD7L from The Hong Kong Polytechnic University.

\bibliographystyle{siam}      
\bibliography{mybib}   

\begin{thebibliography}{10}

\bibitem{BP93}
{\sc J.~R.~S. Blair and B.~Peyton}, {\em An introduction to chordal graphs and
  clique trees}, in Graph Theory and Sparse Matrix Computation, A.~George,
  J.~R. Gilbert, and J.~W.~H. Liu, eds., New York, NY, 1993, Springer New York,
  pp.~1--29.

\bibitem{cifuentes2020geometry}
{\sc D.~Cifuentes, C.~Harris, and B.~Sturmfels}, {\em The geometry of
  \text{SDP}-exactness in quadratic optimization}, Mathematical programming,
  182 (2020), pp.~399--428.

\bibitem{de2011lasserre}
{\sc E.~De~Klerk and M.~Laurent}, {\em On the \text{L}asserre hierarchy of
  semidefinite programming relaxations of convex polynomial optimization
  problems}, SIAM Journal on Optimization, 21 (2011), pp.~824--832.

\bibitem{Nie07JPAA}
{\sc J.~Demmel, J.~Nie, and V.~Powers}, {\em Representations of positive
  polynomials on noncompact semialgebraic sets via \text{KKT} ideals}, Journal
  of Pure and Applied Algebra, 209 (2007), pp.~189--200.

\bibitem{GrimmNetzerSchweighofer07}
{\sc D.~Grimm, T.~Netzer, and M.~Schweighofer}, {\em A note on the
  representation of positive polynomials with structured sparsity}, Archiv der
  Mathematik, 89 (2007), pp.~399--403.

\bibitem{henrion2005detecting}
{\sc D.~Henrion and J.-B. Lasserre}, {\em Detecting global optimality and
  extracting solutions in \text{GloptiPoly}}, in Positive polynomials in
  control, Springer, 2005, pp.~293--310.

\bibitem{henrion2009gloptipoly}
{\sc D.~Henrion, J.-B. Lasserre, and J.~L{\"o}fberg}, {\em Glopti\text{P}oly 3:
  moments, optimization and semidefinite programming}, Optimization Methods \&
  Software, 24 (2009), pp.~761--779.

\bibitem{hua2021exactness}
{\sc Z.~Hua and Z.~Qu}, {\em On the exactness of \text{L}asserre’s relaxation
  for polynomial optimization with equality constraints}, arXiv preprint
  arXiv:2110.13766,  (2021).

\bibitem{Kojima09}
{\sc M.~Kojima and M.~Muramatsu}, {\em A note on sparse {SOS} and {SDP}
  relaxations for polynomial optimization problems over symmetric cones},
  Computational Optimization and Applications, 42 (2009), pp.~31--41.

\bibitem{Lasserre2001}
{\sc J.~B. Lasserre}, {\em Global optimization with polynomials and the problem
  of moments}, SIAM Journal on Optimization, 11 (2001), pp.~796--817.

\bibitem{Lasserre06}
\leavevmode\vrule height 2pt depth -1.6pt width 23pt, {\em Convergent
  {SDP}‐relaxations in polynomial optimization with sparsity}, SIAM Journal
  on Optimization, 17 (2006), pp.~822--843.

\bibitem{lasserre2015introduction}
{\sc J.~B. Lasserre}, {\em An introduction to polynomial and semi-algebraic
  optimization}, vol.~52, Cambridge University Press, 2015.

\bibitem{lasserre2018moment}
{\sc J.~B. Lasserre}, {\em The moment-\text{SOS} hierarchy}, in Proceedings of
  the International Congress of Mathematicians: Rio de Janeiro 2018, World
  Scientific, 2018, pp.~3773--3794.

\bibitem{laurent2009sums}
{\sc M.~Laurent}, {\em Sums of squares, moment matrices and optimization over
  polynomials}, in Emerging applications of algebraic geometry, Springer, 2009,
  pp.~157--270.

\bibitem{magron2021tssos}
{\sc V.~Magron and J.~Wang}, {\em \text{TSSOS}: a \text{J}ulia library to
  exploit sparsity for large-scale polynomial optimization}, arXiv preprint
  arXiv:2103.00915,  (2021).

\bibitem{magron2023sparse}
\leavevmode\vrule height 2pt depth -1.6pt width 23pt, {\em Sparse polynomial
  optimization: theory and practice}, World Scientific, 2023.

\bibitem{newton2022sparse}
{\sc M.~Newton and A.~Papachristodoulou}, {\em Sparse polynomial optimisation
  for neural network verification}, arXiv preprint arXiv:2202.02241,  (2022).

\bibitem{nie2013certifying}
{\sc J.~Nie}, {\em Certifying convergence of lasserre’s hierarchy via flat
  truncation}, Mathematical Programming, 142 (2013), pp.~485--510.

\bibitem{NieReal}
\leavevmode\vrule height 2pt depth -1.6pt width 23pt, {\em Polynomial
  optimization with real varieties}, SIAM Journal on Optimization, 23 (2013),
  pp.~1634--1646.

\bibitem{nie2014ATKM}
\leavevmode\vrule height 2pt depth -1.6pt width 23pt, {\em The $\mathcal{A}$
  truncated $\mbox{K}$-moment problem}, Foundations of Computational
  Mathematics, 14 (2014), pp.~1243--1276.

\bibitem{Nie14}
\leavevmode\vrule height 2pt depth -1.6pt width 23pt, {\em Optimality
  conditions and finite convergence of {L}asserre's hierarchy}, Mathematical
  Programming, 146 (2014), pp.~97--121.

\bibitem{nie2015linear}
\leavevmode\vrule height 2pt depth -1.6pt width 23pt, {\em Linear optimization
  with cones of moments and nonnegative polynomials}, Mathematical Programming,
  153 (2015), pp.~247--274.

\bibitem{nie2019tight}
\leavevmode\vrule height 2pt depth -1.6pt width 23pt, {\em Tight relaxations
  for polynomial optimization and \text{L}agrange multiplier expressions},
  Mathematical Programming, 178 (2019), pp.~1--37.

\bibitem{nie2023moment}
\leavevmode\vrule height 2pt depth -1.6pt width 23pt, {\em Moment and
  Polynomial Optimization}, SIAM, 2023.

\bibitem{nie2009sparse}
{\sc J.~Nie and J.~Demmel}, {\em Sparse sos relaxations for minimizing
  functions that are summations of small polynomials}, SIAM Journal on
  Optimization, 19 (2009), pp.~1534--1558.

\bibitem{NieDemmelSturmfels06}
{\sc J.~Nie, J.~Demmel, and B.~Sturmfels}, {\em Minimizing polynomials via sum
  of squares over the gradient ideal}, Mathematical Programming, 106 (2006),
  pp.~587--606.

\bibitem{NieTang21}
{\sc J.~Nie and X.~Tang}, {\em Convex generalized nash equilibrium problems and
  polynomial optimization}, Mathematical Programming, 198 (2023),
  pp.~1485--1518.

\bibitem{nie2020nash}
\leavevmode\vrule height 2pt depth -1.6pt width 23pt, {\em Nash equilibrium
  problems of polynomials}, Mathematics of Operations Research,  (2023).

\bibitem{nie2021lagrange}
{\sc J.~Nie, L.~Wang, J.~J. Ye, and S.~Zhong}, {\em A \text{L}agrange
  multiplier expression method for bilevel polynomial optimization}, SIAM
  Journal on Optimization, 31 (2021), pp.~2368--2395.

\bibitem{nie2018complete}
{\sc J.~Nie, Z.~Yang, and X.~Zhang}, {\em A complete semidefinite algorithm for
  detecting copositive matrices and tensors}, SIAM Journal on Optimization, 28
  (2018), pp.~2902--2921.

\bibitem{nie2021saddle}
{\sc J.~Nie, Z.~Yang, and G.~Zhou}, {\em The saddle point problem of
  polynomials}, Foundations of Computational Mathematics,  (2021), pp.~1--37.

\bibitem{Putinar93}
{\sc M.~Putinar}, {\em Positive polynomials on compact semi-algebraic sets},
  Indiana University Mathematics Journal, 42 (1993), pp.~969--984.

\bibitem{reznick2000some}
{\sc B.~Reznick}, {\em Some concrete aspects of {H}ilbert's 17th problem},
  Contemporary mathematics, 253 (2000), pp.~251--272.

\bibitem{waki2006sums}
{\sc H.~Waki, S.~Kim, M.~Kojima, and M.~Muramatsu}, {\em Sums of squares and
  semidefinite program relaxations for polynomial optimization problems with
  structured sparsity}, SIAM Journal on Optimization, 17 (2006), pp.~218--242.

\bibitem{wang2021certifying}
{\sc J.~Wang and V.~Magron}, {\em Certifying global optimality of {AC-OPF}
  solutions via the \text{CS-TSSOS} hierarchy}, arXiv preprint
  arXiv:2109.10005,  (2021).

\bibitem{TSSOS}
{\sc J.~Wang, V.~Magron, and J.-B. Lasserre}, {\em \text{TSSOS}: A
  \text{Moment-SOS} hierarchy that exploits term sparsity}, SIAM Journal on
  Optimization, 31 (2021), pp.~30--58.

\bibitem{CSTSSOS}
{\sc J.~Wang, V.~Magron, J.~B. Lasserre, and N.~H.~A. Mai}, {\em {CS-TSSOS}:
  {C}orrelative and term sparsity for large-scale polynomial optimization}, ACM
  Transactions on Mathematical Software, 48 (2022), pp.~1--26.

\end{thebibliography}

\end{document}